\documentclass[letterpaper,11pt]{article}

\usepackage{etoolbox}
\newtoggle{anonymous}
\togglefalse{anonymous}

\usepackage{amsfonts,amscd,amssymb,amsmath,amsthm}
\usepackage{centernot,caption,subcaption,color,url}
\usepackage[T1]{fontenc}
\usepackage{authblk}
\usepackage{graphicx}
\usepackage{caption}
\usepackage{subcaption}
\usepackage[colorlinks=true,linkcolor=blue,citecolor=red]{hyperref}
\usepackage{hyperref}
\usepackage[capitalize]{cleveref}

\usepackage[margin=1in]{geometry}
\usepackage{setspace}
\newtheorem{theorem}{Theorem}[section]
\newtheorem{lemma}[theorem]{Lemma}

\newtheorem{claim}[theorem]{Claim}
\newtheorem{question}{Question}

\theoremstyle{definition}

\newtheorem{remark}[theorem]{Remark}

\title{Reconstruction of graph colourings}

\author{Yu. Demidovich, Ya. Panichkin, M. Zhukovskii}
\author{Yury Demidovich\thanks{AI Initiative, KAUST, Kingdom of Saudi Arabia; \texttt{yury.demidovich@kaust.edu.sa}} \,\,\,\, Yaroslav Panichkin\thanks{Moscow Institute of Physics and Technology, Russia; \texttt{panichkin.yak@gmail.com}} \,\,\,\,    Maksim Zhukovskii\thanks{Department of Computer Science, University of Sheffield, UK; \texttt{m.zhukovskii@sheffield.ac.uk} }
 }

\date{}

\begin{document}

\maketitle

\begin{abstract}

A $k$-deck of a (coloured) graph is a multiset of its induced $k$-vertex subgraphs. Given a graph $G$, when is it possible to reconstruct with high probability a uniformly random colouring of its vertices in $r$ colours from its $k$-deck? In this paper, we study this question for grids and random graphs. 

Reconstruction of random colourings of $d$-dimensional $n$-grids from the deck of their $k$-subgrids ($k\times\ldots\times k$ grids) is one of the most studied colour reconstruction questions. The 1-dimensional case is motivated by the problem of reconstructing DNA sequences from their `shotgunned' stretches. It was comprehensively studied and the above reconstruction question was completely answered in the '90s. In this paper, we get a very precise answer for higher $d$. For every $d\geq 2$ and every $r\geq 2$, we present an almost linear algorithm that reconstructs with high probability a random $r$-colouring
of vertices of a $d$-dimensional $n$-grid from the deck of all its $k$-subgrids for every $k\geq(d\log_r n)^{1/d}+1/d+\varepsilon$. We also prove that the random $r$-colouring is not reconstructible with high probability if $k\leq (d\log_r n)^{1/d}-\varepsilon$. This answers the question of Narayanan and Yap (that was asked for $d\geq 3$) on ``two-point concentration'' of the minimum $k$ so that $k$-subgrids determine the entire colouring.

Next, we prove that with high probability a uniformly random $r$-colouring of the vertices of a uniformly random graph $G(n,1/2)$ is reconstructible from its full $k$-deck if $k\geq 2\log_2 n+8$ and is not reconstructible with high probability if $k\leq\sqrt{2\log_2 n}$. We further show that the colour reconstruction algorithm for random graphs can be modified and used for graph reconstruction: we prove that with high probability $G(n,1/2)$ is reconstructible from its full $k$-deck if $k\geq 2\log_2 n+11$ (while it is not reconstructible with high probability if $k\leq 2\sqrt{\log_2 n}$).

\end{abstract}




\section{Introduction}

The problem of reconstructing global properties of discrete structures from their substructures (in particular, from `local' observations) is widely studied (see, e.g.,~\cite{Intro_ACKR,Intro_BHZ,Intro_BH,Intro_FT,Intro_Peb,Intro_RS}) and has applications, e.g., in graph isomorphism problem~\cite{Babai}, in DNA sequencing~\cite{DNA,DNA2,DNA3}, and recovering neural networks~\cite{Soudry}. One of the most famous open problems in this spirit is the reconstruction conjecture of Kelly and Ulam~\cite{K2,K1,Ulam}: any graph on $n\geq 3$ vertices can be reconstructed from the multiset of its unlabelled induced $(n-1)$-subgraphs. The inspiring recent work of Mossel and Ross~\cite{RM} renewed the interest to this topic by introducing the problem of graph shotgun assembly. Since then there has been extensive study on various shotgun assembly questions~\cite{Shotgun_BBN,Shotgun_BFM,Shotgun_GM,Shotgun_JKRS} and, in particular, shotgun assembly for vertex-colourings~\cite{DL,NP,PRS}. 

In this paper, we study the problem of reconstruction of random colourings of vertices of graphs from their (small) subgraphs. Let us formally state the question. Let $G$ be a simple graph on $n$ vertices, and $r\geq 2$, $k\geq 1$ be integers. Assign uniformly at random $r$ colours to the vertices of $G$: every vertex is coloured in one of $r$ colours uniformly at random independently of all the others. A {\it $k$-deck of $G$} is a multiset of (not necessarily all) its $k$-vertex (unlabelled) {\it coloured} induced subgraphs. 

\begin{question}
Is it possible to reconstruct the random $r$-colouring $C$ of $G$ from its given $k$-deck $\mathcal{D}$ with high probability\footnote{With probability approaching 1 as $n\to\infty$. In what follows, we will write simply `whp' for brevity.}? 
\label{question}
\end{question}

In other words, is it true that whp, for any {\it other} $r$-colouring $C'$ of $G$ (i.e. such that the $C$-coloured $G$ and  $C'$-coloured $G$ are not isomorphic), the multiset of all $C'$-coloured induced subgraphs of $G$ does not contain $\mathcal{D}$ as a submultiset? We address this question for $d$-dimensional grids, which is one of the most studied shotgun assembly questions, and random graphs.

\subsection{Grids}

Motivated by the problem of reconstructing DNA sequences from their `shotgunned' stretches, Arratia, Martin, Reinert and Waterman~\cite{DNA} and Dyer, Frieze and Suen~\cite{DNA2} answered Question~\ref{question} for a 1-dimensional $n$-lattice $G=P_n$ and the deck consisting of all subpaths of length $k$. If $k=2\log_r n-\omega(1)$, then whp (as $n\to\infty$) it is impossible to reconstruct the colouring of $G$; if $k=2\log_r n+\omega(1)$, then whp the colouring of $G$ is reconstructible. This is tight since for $k=2\log_r n+O(1)$ the limit probability of reconstructibility is non-trivial. 

Note that, for $k=\log_r n-\omega(1)$, which is less than a half of the reconstruction threshold $2\log_r n$, the non-reconstructibility statement is straightforward. Indeed, since the deck has exactly $n-k+1$ paths and there are at most $r^k$ differently coloured $k$-paths, we get that the number of different decks is at most ${n+r^k\choose r^k}$. On the other hand, the number of non-isomorphic colourings of $P_n$ is at least $\frac{1}{2}r^{n}$. Since $2{n+r^k\choose r^k}r^{-n}=o(1)$ whenever $k=\log_r n-\omega(1)$, we immediately get the 0-statement. Improving the lower bound for the threshold function for reconstructibility by a factor of 2 appears to be possible since whp there are non-overlapping $k$-subgrids $x_1,x_2,y_1,y_2$ such that intervals between $x_1,x_2$ and between $y_1,y_2$ do not overlap, and, for every $j\in\{1,2\}$, $x_j$ and $y_j$ are isomorphic. So the intervals between $x_1,x_2$ and between $y_1,y_2$ can be swapped.


Similar swaps are not possible whp in higher dimensions since they require the entire `frames' around swapped intervals to be isomorphic. On the other hand, a similar simple counting argument (see Section~\ref{sc:grids_proof_l}) can be applied to get the following 0-statement for the $d$-dimensional $n$-lattice 
$$
G=H^d_n:=\underbrace{P_n\Box\ldots\Box P_n}_{d}
$$ 
(as usual, $\Box$ denotes the cartesian product of graphs) and the deck consisting of all its $d$-dimensional $k$-subgrids: if $k=(d\log_r n)^{1/d}-\Omega(1)$, then whp the colouring of $G$ is not reconstructible. Quite surprisingly, in contrast to the 1-dimensional case, this bound appears to be sharp up to an additive constant term, as we explain in the next paragraph.


In~\cite{RM}, Mossel and Ross conjectured that there exists a reconstructibility threshold, i.e.\ a function $k_{d,r}=k_{d,r}(n)$ such that, for every $\varepsilon>0$, the random $r$-colouring of $G$ is not reconstructible whp if $k<(1-\varepsilon)k_{d,r}$ and 
 reconstructible whp if $k>(1+\varepsilon)k_{d,r}$. Ding and Liu~\cite{DL} resolved the conjecture and proved that $k_{d,r}=(d\log_r n)^{1/d}$. Narayanan and Yap~\cite{NP} proved the ``two-point concentration'' for $d=2$: if $k\geq (2\log_r n)^{1/2}+\frac{3}{4}$, then the colouring of $G$ is reconstructible whp, and if $k< (2\log_r n)^{1/2}-\frac{1}{4}$, then the colouring of $G$ is not reconstructible whp. They noted that a modification of their argument may give reasonable bounds on the threshold, though to get ``two-points concentration'', a higher-dimensional generalisation of so called ``interface paths'' seems {\it necessary} to be developed. In this paper, we establish the ``two-point concentration'' in all dimensions $d\geq 2$, and present an almost linear algorithm that whp reconstructs a random colouring from the $k$-deck, for any $k$ above the threshold.



\begin{theorem}
Let $d\geq 2$ and $\varepsilon>0$.
\begin{enumerate}
    \item If $k\leq (d\log_r n)^{1/d}-\varepsilon$, then whp the uniformly random $r$-colouring of $H^d_n$ is not reconstructible.
    \item If $k\geq(d\log_r n)^{1/d}+\frac{1}{d}+\varepsilon$, then whp the uniformly random $r$-colouring of $H^d_n$ is reconstructible and there exists a linearithmic-time\footnote{The running time is $O(N\log N)$, where $N=(n-k+1)^d k^d$ is the input size.} algorithm that reconstructs the colouring of $H^d_n$.
\end{enumerate} 
\label{th:grids}
\end{theorem}

We shall note that the linearithmic-time algorithm is a derandomised version of a randomised algorithm that reconstructs the colouring in linearithmic time 
whp (in product measure) as well. Though a straightforward derandomisation requires an extra $|V(G)|$-factor, we show that the colouring of the deck contains enough randomness to get input (pseudo)random bits from it --- see details in Section~\ref{sc:grids_time}. For the sake of simplicity and clarity of presentation we however present the randomised algorithm (instead of presenting its derandomised version straight away), and then explain why the entire proof of the algorithm's success works well for the (pseudo)random\footnote{It is actually random in the sense that it is a function of the random colouring $C$ and a labelling of the deck.} input. Let us also mention that our algorithm can be used to reconstruct randomly coloured tori $\mathbb{Z}_n^d$ in contrast to the algorithm of Ding and Liu~\cite{DL} --- see Remark~\ref{rmk:tori}.

Finally, note that in~\cite{DNA,DL,DNA2,NP} it is assumed that the orientation of $k$-subgrids is observed (i.e. there are exactly $r^{k^d}$ non-isomorphic coloured subgrids). For consistency with the previous study, in our proof of Theorem~\ref{th:grids} in Section~\ref{sc:grids_proof} we make the same assumption. Nevertheless, the `unoriented' case can be treated similarly. We discuss such a modification in Remark~\ref{rmk:unoriented}.

\paragraph{Proof strategy.}

Most steps of our randomised algorithm are similar to those in the randomised algorithm suggested by Narayanan and Yap (though they did not prove that it admits a derandomisation and that the successful run of this algorithm implies reconstruction in the strong sense --- the fact that the algorithm outputs the initial colouring does not necessarily imply that the colouring with the given deck is unique, see Section~\ref{sc:verification}): start from a random $k$-subgrid from the deck and then, at each step, extend the reconstructed lattice by a single subgrid that shares with the previous subgrid a $(k-1)\times k$ rectangle. Depending on the current position, there are three ways to do an extension: naive extension, corner extension and internal extension. Corner and internal extensions allow to reduce the probability that the algorithm rejects by looking ahead: find in the deck a bunch of subgrids that form a large enough rectangle that extends the previous subgrid. If there is a unique sequence of such subgrids, then the algorithm does not reject. The above mentioned ``interface paths'' allow to bound from above the probability of rejection for corner extensions. We manage to get rid of the corner extensions (and thus avoid a generalisation of the concept of ``interface path'') by showing, roughly, that it is not likely that a randomly coloured lattice contains a short cycle consisting of rectangles such that the `right-end' of every rectangle is isomorphic to the `left-end' of its successor in the cycle. Below, we state this claim. It is also worth mentioning that our reconstruction algorithm is slightly faster since in corners we ``look ahead'' only $k$ subgrids while the algorithm of Narayanan and Yap explores $k^2$ subgrids.

For convenience, let us assume that the vertices of $G$ are elements of $[n]^d$ that are adjacent whenever the distance between them is exactly 1. Let us label all $k$-subgrids of $G$ by their smallest element. In particular, for any two subgrids $x,y\in[n-k+1]^d$, there exists $\mathbf{w}\in\mathbb{Z}^d$ such that $x+\mathbf{w}=y$. For any $x\in[n-k+1]^d$ and $\mathbf{w}\in\mathbb{Z}^d$, we let $x[\mathbf{w}]=x\cap(x+\mathbf{w})$ to be the subgrid of $x$ consisting of vertices that also belong to $x+\mathbf{w}$. Set $\mathbf{e}=(1,0,\ldots,0)\in\mathbb{Z}^d$. Consider an auxiliary directed graph $\mathcal{G}$ on $[n-k+1]^d$ with blue and red edges defined as follows. Let us draw a blue edge from $x$ to $y$ if $x+\mathbf{e}=y$. Clearly, blue edges constitute a disjoint union of directed $(n-k+1)$-paths. We draw a red edge from $x$ to $y$ if $x[\mathbf{e}]$ is isomorphic to $y[-\mathbf{e}]$ (as a coloured graph) and there is no blue edge from $x$ to $y$. A {\it rainbow path} in $\mathcal{G}$ is a (directed) path comprising both blue and red edges. 
 For $x\in[n-k+1]^d$, let $\mathcal{N}(x)$ be the {\it $G$-neighbourhood} of $x$, i.e. the set of all subgrids $y\in[n-k+1]^d$ that have common vertices with $x$ in $G$. 

\begin{lemma}
For any $k\geq(d\log_r n)^{1/d}+\frac{1}{d}+\varepsilon$ whp in $\mathcal{G}$ there are no rainbow paths $x_1x_2\ldots x_{\ell}$ with $\ell\leq 3k$ vertices such that $x_{\ell}\in\mathcal{N}(x_1)$.
\label{lm:lattices_main}
\end{lemma}

Another advantage of Lemma~\ref{lm:lattices_main} is that it makes it possible to prove the linearithmic time bound in Theorem~\ref{th:grids} --- see the proof of Claim~\ref{cl:few_wrong_paths} in Section~\ref{sc:grids_time}.\\

We prove Lemma~\ref{lm:lattices_main} in Section~\ref{sc:lm_proof}, and then describe an efficient algorithm of reconstructing colours and prove the 1-statement in Theorem~\ref{th:grids} in Section~\ref{sc:grids_proof_u}.

\subsection{Random graphs}
\label{intro:rg_colour}

For an arbitrary graph $G$, an answer on Question~\ref{question} strongly depends on the group of automorphisms of $G$ denoted by $\mathrm{Aut}(G)$. In particular, if $\mathrm{Aut}(G)=\mathrm{Sym}(V(G))$ is the full symmetric group (this happens if and only if $G$ is either complete or empty), then, obviously, {\it any} colouring of $G$ is reconstructible from the {\it full} 1-deck (i.e. from the set of all coloured vertices), since it is enough to know the total number of vertices coloured in each of the colours. On the other hand, if $G=G(n)$ is a sequence of asymmetric graphs ($\mathrm{Aut}(G)$ is trivial) on $[n]:=\{1,\ldots,n\}$, then the minimum $k$ such that a random $r$-colouring of $G$ is reconstructible whp from its full $k$-deck admits a non-trivial lower bound. Indeed, let $k<\sqrt{2\log_2 n}$. Since asymptotically almost all labelled graphs are asymmetric~\cite{ER}, there are at most $\frac{1+o(1)}{k!}2^{{k\choose 2}}$ graph isomorphism classes presented in the full $k$-deck, and then the number of non-isomorphic coloured graphs in the $k$-deck is at most $\frac{(er)^k}{k^k} 2^{{k\choose 2}}$ for large enough $k$. Therefore, the number of different decks is at most 
$$
F_k:=
{
{n\choose k}+
\left\lceil\frac{(er)^k}{k^k} 2^{{k\choose 2}}\right\rceil
\choose 
\left\lceil\frac{(er)^k}{k^k} 2^{{k\choose 2}}
\right\rceil}.
$$
On the other hand, the number of non-isomorphic colourings of $G$ is exactly $r^n$ due to asymmetry of $G$. Since $F_k r^{-n}=o(1)$, we get that whp the colouring of $G$ is not reconstructible.

There are asymmetric graphs such that their random colourings are not reconstructible from their full $k$-decks whp for a certain $k=(2-o(1))\log_r n$. Indeed, let us make $P_{n-1}=v_1\ldots v_{n-1}$ asymmetric by, say, adding $v_n$ adjacent to $v_{n-3}$ and $v_{n-2}$. But then the proof of non-reconstructibility is the same as for $P_n$: for $k<(2-\varepsilon)\log_r n$, whp we can find in our graph $k$-paths
$x_1,x_2,y_1,y_2$ such that paths $(x_1\ldots x_2)$ and $(y_1\ldots y_2)$ between $x_1,x_2$ and between $y_1,y_2$ respectively do not overlap, and $x_j\cong y_j$, $j\in\{1,2\}$.
 Then, if we do the swap $(x_1\ldots x_2)\leftrightarrow (y_1\ldots y_2)$, the modified graph and the initial one would have equal full $k$-decks but would not be isomorphic. Though full $k$-decks contain disconnected subgraphs as well, they would not prevent the coincidence of the decks. 


Whenever $k=o(\log n)$, we do not have an example of an asymmetric $G$ such that its colouring is whp reconstructible from its full $k$-deck. Nevertheless, we manage to  prove that, for asymptotically almost all labelled graphs (and, thus, for asymptotically almost all labelled asymmetric graphs), their colourings are whp reconstructible from their $\lceil 2\log_2 n+8\rceil $-decks. As usual, we denote by $G(n,p)$ the binomial random graph with the probability of appearance of an edge equal to $p$. Everywhere below, when we deal with random colourings of random graphs, we consider the product measure.

\begin{theorem}
Let $r\geq 2$.
\begin{enumerate}
\item If $k\leq \sqrt{2\log_2 n}$, then, for any sequence of asymmetric graphs $G=G(n)$ on $[n]$, whp its random $r$-colouring is not reconstructible from its full $k$-deck.
\item If $k\geq 2\log_2 n+8$, then whp the random $r$-colouring of $G(n,1/2)$ is reconstructible from its full $k$-deck.
\end{enumerate}
\label{th:random_graphs_colour}
\end{theorem}

The 0-statement is already proven. We prove the 1-statement in Section~\ref{sc:random_graphs_colour_proof} by introducing a colour reconstruction algorithm that actually does not use the knowledge about the input graph. Thus, a modification of this algorithm can be used for the problem of graph reconstruction that we discuss in the next section.

\paragraph{Proof strategy.} 

The colour reconstruction algorithm (Algorithm A) we introduce in Section~\ref{sc:random_graphs_colour_proof} heavily relies on the fact that whp in randomly coloured $G(n,1/2)$ there is an induced path of the maximum length such that its colouring differs from colourings of all the other induced paths. Thus, it is possible to find in the graphs from the deck such a unique induced path $P$ just by counting the number of graphs from the deck it belongs to. Moreover, we show that there are no three vertices in $G(n,1/2)$ that have exactly the same neighbourhoods in the unique path $P$ (or its image under a unique non-trivial automorphism, if it exists). It means that, in every graph $D$ from the deck that contains $P$, all vertices from $D\setminus P$ could be divided into equivalence classes of sizes at most 2 with respect to their neighbourhoods in $P$. Omitting some technical details, all that remains is to distinguish pairs of vertices in equivalence classes of size 2. This appears to be possible since whp in $G(n,1/2)$ differences of neighbourhoods for pairs of vertices are sufficiently large.

\subsection{Graph reconstruction}

The problem of reconstructing colours is closely related to the well-known {\it graph reconstruction} problem. The famous reconstruction conjecture of Kelly and Ulam~\cite{K2,K1,Ulam} asserts that every graph $G$ with $n\geq 3$ vertices can be reconstructed (up to isomorphism) from the multiset of all its unlabelled induced $(n-1)$-subgraphs. M\"{u}ller~\cite{Muller} and Bollob\'{a}s~\cite{Bol_reconstruction} showed that the conjecture holds for asymptotically almost all graphs, i.e. whp holds for the binomial random graph $G_n\sim G(n,1/2)$. 
Moreover, M\"{u}ller~\cite{Muller} and Spinoza and West~\cite{SW} proved that, for any $\varepsilon>0$ and any integer $k\geq (1+\varepsilon)\frac{n}{2}$, whp $G_n$ is reconstructible from its $k$-deck. Finally, it was proved by Pikhurko~\cite{Pikhurko} that the latter bound on the reconstruction threshold can be siginfincantly improved: for any $\varepsilon>0$ and any integer $k\geq (2+\varepsilon)\log_2n$, whp $G_n$ is reconstructible from its $k$-deck. Unfortunately, this result was never published. For a more complete survey on the graph reconstruction we refer a reader to~\cite{KW_survey}. 

Let us now build a bridge between colour reconstruction and graph reconstruction.

\begin{lemma}
If, for any sequence of asymmetric graphs $G=G(n)$ on $[n]$, its uniformly random 2-colouring is reconstructible from the full (coloured) $k(n)$-deck with probability $o(n^{-1/2})$, then whp $G(n,1/2)$ is not reconstructible from its full (uncoloured) $k(2n)$-deck. 
\label{lm:reconstruction_bridge}
\end{lemma}

Lemma~\ref{lm:reconstruction_bridge} is proven in Section~\ref{sc:bridge_lemma_proof}.\\

Set $k(n)=\left\lfloor \sqrt{2\log_2 n}\right\rfloor$. Note that the probability bound in the 0-statement in Theorem~\ref{th:random_graphs_colour} can be easily specified: since $F_{k(n)} 2^{-n}=o(n^{-1/2})$, we get that a random $2$-colouring of an asymmetric $G$ is reconstructible with probability $o(n^{-1/2})$. Then, by Lemma~\ref{lm:reconstruction_bridge}, we get that whp $G_n$ is not reconstructible from its full $k(2n)$-deck, i.e. whp $G_n$ is not reconstructible from the $\left\lfloor\sqrt{2\log_2 n+2}\right\rfloor$-deck.
Though, as we will see below, the usual counting argument applied directly to the reconstruction problem gives a slightly better lower bound for the reconstruction threshold, Lemma~\ref{lm:reconstruction_bridge} may be beneficial for the reconstruction threshold if one may improve the lower bound in Theorem~\ref{th:random_graphs_colour} by at least a factor of 2.

Let us apply the counting argument. Let $k\leq 2\sqrt{\log_2 n}$. As we mentioned in Section~\ref{intro:rg_colour}, there are at most $\frac{1+o(1)}{k!}2^{{k\choose 2}}$ graph isomorphism classes presented in the full $k$-deck, and then the number of different decks is at most 
$$
\tilde F_k:=
{
{n\choose k}+
\left\lceil (e/k)^k 2^{{k\choose 2}}\right\rceil
\choose 
\left\lceil (e/k)^k 2^{{k\choose 2}}
\right\rceil}.
$$
On the other hand, the number of unlabelled graphs is at least $\frac{1}{n!}2^{{n\choose 2}}$. Since $n!\tilde F_k 2^{-{n\choose 2}}=o(1/n!)$, we get that whp $G_n$ is not reconstructible. Moreover, in Section~\ref{sc:random_graphs_rec_proof}, we significantly improve the upper bound of M\"{u}ller and Spinoza and West on the reconstruction threshold which has been best known published result. Our upper bound is fairly close to the obtained lower bound.

\begin{theorem}
$\,$
\begin{enumerate}
\item If $k\leq 2\sqrt{\log_2 n}$, then whp $G(n,1/2)$ is not reconstructible from its full $k$-deck. 
\item If $k\geq 2\log_2 n+11$, then whp $G(n,1/2)$ is reconstructible from its full $k$-deck. 
\end{enumerate}
\label{th:random_graphs_rec}
\end{theorem}

\paragraph{Proof strategy.}

Though the bound from~\cite{Pikhurko} is asymptotically close to our bound,
the reconstruction from the full $k$-deck for $k\geq (2+\varepsilon)\log_2 n$ is reasonably easier. Indeed, by the union bound, for every isomorphism class of graphs of size $k$ as above, whp $G_n$ does not contain a representative of this class, which is not the case for $k=2\log_2n +O(1)$. Fix a $k$-set and observe that whp it induces an asymmetric subgraph $H$ such that all its induced subgraphs of size at least $k/2+\varepsilon$ are non-isomorphic and asymmetric as well. From this it can be derived that there are no other subgraphs in $G_n$ isomorphic to $H$. Then we can extract a $k$-subgraph with such properties from the full $(k+2)$-deck. It is easy to see that it helps to distinguish between all the other vertices and reconstruct adjacencies between them. Removing the $\varepsilon$-term requires analysing specific {\it explicit} asymmetric maximum subgraphs in $G_n$. For that, we use a similar strategy as in the proof of the colour reconstruction (Theorem~\ref{th:random_graphs_colour}). Actually, the algorithm for colour reconstruction does not use the input graph $G$ and reconstructs $G$ simultaneously with its colouring. However, for the graph reconstruction we can not use colours to identify a specific maximum induced path. We overcome this by showing that there is a maximum induced path $P$ and a 3-tuple of vertices $\mathbf{u}$ such that  the neighbourhood of $\mathbf{u}$ in $P$ is unique, see details in Section~\ref{sc:random_graphs_rec_proof}.


\section{Proof of Theorem~\ref{th:grids}}
\label{sc:grids_proof}

We let $G=H_n^d$ to be a $d$-dimensional $n$-lattice, and let $\mathcal{D}_k(G)$ stand for the deck of randomly $r$-coloured $G$ comprising all its $k$-subgrids.

\subsection{Lower bound}
\label{sc:grids_proof_l}
Let $\varepsilon>0$ be a small constant. Consider any positive integer $k\leq (d\log_r n)^{1/d} - \varepsilon$. We shall prove that whp the random colouring $C$ of $G$ is not reconstructible from $\mathcal{D}_k(G)$. Note that
\begin{align}
    r^{k^d} \log_r (n^d) & \leq r^{\left((d\log_r n)^{1/d} - \varepsilon\right)^d}\log_r (n^d) \notag\\
    & = dr^{d\log_r n  -\varepsilon d(d\log_r n)^{(d-1)/d}(1+o(1))}\log_r n \notag\\
    & = dn^dr^{-\varepsilon d(d\log_r n)^{(d-1)/d}(1+o(1))}\log_r n \notag\\
    & =n^d \exp\left[-\Theta\left(\log^{1-1/d} n\right)
    \right]=o(n^d)\label{eq:thm_thr_aux}.
\end{align}

Observe that the probability that the random colouring of the graph is reconstructible from its $k$-deck is at most
\begin{align*}        \sum_{\mathcal{D}}\mathbb{P}\left(\mathcal{D}_k\left(G\right) = \mathcal{D}\;\text{ and }\;C \text{ is reconstructible from }\mathcal{D}\right) & \leq \sum_{\mathcal{D}} \frac{1}{r^{n^d}}\\
    & = {r^{k^d}+(n-k+1)^d-1\choose (n-k+1)^d}r^{-n^d},
\end{align*}
where the sum is over all possible decks $\mathcal{D}$; the number of decks equals ${r^{k^d}+(n-k+1)^d-1\choose (n-k+1)^d}.$ Due to \eqref{eq:thm_thr_aux}, we get $r^{k^d}<(n-k+1)^d<n^d$. Therefore,
\begin{align*}
    {r^{k^d}+(n-k+1)^d-1\choose (n-k+1)^d}& \leq \frac{(2n^d)^{r^{k^d}}}{(r^{k^d})!} \leq\left(\frac{2en^d}{r^{k^d}}\right)^{r^{k^d}}
    \leq n^{d\cdot r^{k^d}}\\
    &=r^{r^{k^d}\log_r(n^d)}\stackrel{\eqref{eq:thm_thr_aux}}{=}r^{o(n^d)}=o(r^{n^d}).
\end{align*}
Thus, the probability of reconstruction approaches 0, concluding the 0-statement.

\subsection{Proof of Lemma~\ref{lm:lattices_main}}
\label{sc:lm_proof} 


We shall prove that whp there are no {\it inclusion-minimal} rainbow paths $x_1\ldots x_{\ell}$ in $\mathcal{G}$ with $x_{\ell}\in\mathcal{N}(x_1)$ (minimality means that we exclude subpaths $x_i\ldots x_j$, where $x_j\in\mathcal{N}(x_i)$, $\ell\ge j>i\geq 1$). It would obviously imply the desired assertion. Let us bound the expected number of such paths. First we choose the number of vertices in the path $\ell\leq 3k$. Then we choose the number of red edges $h\leq\ell-1$ and their positions in ${\ell-1\choose h}$ ways. Let us bound the number of choices of $h+1$ blue subpaths (here we assume that a blue subpath may be empty, i.e. consisting of a single vertex). Since lengths of blue paths are already fixed, the number of choices of the first $h$ paths (all but the last one) is the number of choices of their first vertices, which is at most $n^{d h}$. As for the last path, the number of choices of its last vertex is at most $(2k)^d$, since it must belong to $\mathcal{N}(x_1).$ Let $x_{\nu}$ be the first vertex of the last path. 
Since we count inclusion-minimum paths, we may assume that all $x_i,x_j$, $i<j<\ell$, that are not along the same blue subpath, are disjoint in the $n$-lattice. It immediately implies that, for every $i\in[\ell-2]$ such that $x_i$ and $x_{i+1}$ are not blue-adjacent, the probability that there is a red edge between them in $\mathcal{G}$ equals 
  $r^{-(k^d-k^{d-1})},$ and these events are independent (over $i$). 
Morever, for the same reason, we may assume that $x_{\ell}$ does not overlap with $x_2,\ldots,x_{\nu-1}$. So, even if $\nu=\ell$, we get that $x_{\nu}$ does not overlap with $x_2,\ldots,x_{\nu-1}$. Thus, the red adjacency of $x_{\nu-1}$ and $x_{\nu}$ is also independent of all the other red adjacencies and has the same probability $r^{-(k^d-k^{d-1})}$ (even when $\nu-1=1$ and $\nu=\ell$).


  


Therefore, the expected number of rainbow paths $x_1\ldots x_{\ell}$ with $x_{\ell}\in\mathcal{N}(x_1)$ is at most
$$
 \sum_{\ell=1}^{3k}\sum_{h=1}^{\ell-1} {\ell-1\choose h} n^{dh} (2k)^d  r^{-h(k^d-k^{d-1})} \leq (2k)^d\sum_{\ell=1}^{3k}\sum_{h=1}^{\ell - 1}\left(3k n^{d}r^{-\left(k^d-k^{d-1}\right)}\right)^h.
$$
Letting $k_{\mathrm{th}} = (d\log_r n)^{1/d}$, we get that
\begin{equation}
 k n^{d} r^{-(k^d-k^{d-1})} \leq  
(k_{\mathrm{th}}+1/d+\varepsilon)n^d r^{-k_{\mathrm{th}}^d-\varepsilon dk_{\mathrm{th}}^{d-1}(1+o(1))}\to 0.
\label{eq:helpful}
\end{equation}
Thus, the expected number does not exceed 
\begin{align*} 
 (2k)^d(3k)^3n^dr^{-(k^d-k^{d-1})}
 &\leq (3k)^{d+3}n^dr^{-(k^d-k^{d-1})}\\
 &\leq \left(3k_{th}+\frac{3}{d}+3\varepsilon\right)^{d+3}r^{-\varepsilon dk_{\mathrm{th}}^{d-1}(1+o(1))}\\
 &=r^{-\varepsilon dk_{\mathrm{th}}^{d-1}(1+o(1))}=o(1).
\end{align*}
It remains to apply Markov's inequality.

\subsection{Upper bound}
\label{sc:grids_proof_u}

In this section, we describe an efficient algorithm that reconstructs a random colouring of $G$ whp and prove the second part of Theorem~\ref{th:grids}.

\subsubsection{Overview}
\label{sc:overview}

We assume that the deck $\mathcal{D}=\mathcal{D}_k(G)$ is ordered in some way. As an intermediate step, we present a randomised algorithm whose outcome, for every sample of random bits, is affected only by the position in $G$ of the first $D\in\mathcal{D}$. Its straightforward derandomisation increases the running time by a factor of $O(n^d)$. Nevertheless, we show in Section~\ref{sc:grids_time} that there exists an efficient derandomisation preserving the linearithmic bound on the running time. Time complexity of the randomised algorithm is analysed in Section~\ref{sc:grids_time} as well.

Assume that an ordering of the deck is fixed as a result of sampling the random bits. Initially (step 0), the algorithm ``places'' the first subgrid from the deck and removes it from the deck. Further, at every iteration the algorithm ``places'' precisely one subgrid to a particular place by merging it with the previously reconstructed partial colouring and removes it from the deck. Note that the ``exploration path'' of the grid is fixed (at each step, the algorithm tries to place a card at specific place fixed in advance unless it reaches the border of the grid). Eventually, the algorithm either outputs a colouring of $G$ (we should still verify that it is the initial colouring, and that there is no other colouring with the same deck) or returns an `ambiguous' value. 

Assume that the algorithm finishes its work and the deck is empty. As we show in Section~\ref{sc:verification}, it means that the output colouring equals $C$ and that it is reconstructible whp. In order to show this, we run the same algorithm multiple times for all the remaining directions of the ``exploration path'' (there are constantly many). 
If all the verification runs successfully output colourings of $G$ (and this is actually the case whp), then, as we show in Section~\ref{sc:verification}, it implies that all the colourings that the algorithm outputs coincide, and there is no other colouring of the grid with the same deck. In other words, the colouring is reconstructible. Let us underline that the verification runs are needed only to prove that $C$ is reconstructible, while the reconstruction algorithm itself requires only a single run.\\

In Section~\ref{sc:step-0}, we show that a uniformly random ordering of the deck (which is equivalent, in terms of the output, to a uniformly random choice of the first element in the deck since the algorithm is designed in a way so that, as soon as the first position in $G$ is fixed, the ``exploration path'' of the colouring of $G$ is fixed as well) results in successful runs of both the algorithm and its verification ``reflections'' whp. Thus, there exists a choice of the first subgrid in the deck so that the algorithm colours the entire grid under derandomisation: in order to reconstruct $G$ we choose elements of the deck on the role of the initial position one by one and run the algorithm until it outputs the entire colouring.

At the $i$th step of a single run, $i\geq 1$, the algorithm explores the remaining deck, finds a subgrid that fits the current position, and checks that either no other subgrid fits, or another suitable subgrid is wrong since its neighbourhood can not be reconstructed, see details in Section~\ref{sc:grid_algorithm}. If this is not the case, then the algorithm returns `ambiguous'. Otherwise, the algorithm ``places'' the detected unique subgrid to the current position, removes it from the deck, and switches to the next step. 




\subsubsection{Step 0 and derandomisation}
\label{sc:step-0}

Consider a uniformly random ordering of $\mathcal{D}$. Let $\mu$ be the first subgrid in this order. At step 0, the algorithm ``places'' this subgrid and removes it from $\mathcal{D}$.

We may consider $\mu$ as a uniformly random vertex of $\mathcal{G}$. For $x\in[n-k+1]^d,$ let $\mathcal{B}(x)$ be the event that $\mu$ is isomorphic to $x$ in the coloured $G$. The algorithm is designed in such a way (see Section~\ref{sc:grid_algorithm}) that as soon as the position of the first subgrid is fixed (i.e. subject to the event $\{\mu=x\}\wedge(\wedge_{y\neq x}\neg \mathcal{B}(y))$), the position of the $i$th subgrid for every $i\geq 1$ in $G$ is determined in a unique way. If $\mu=x$ and $\mathcal{B}(y)$ holds for some $y\neq x$, then the algorithm returns `ambiguous' and gives up.




 We say that {\it $\mu$ is successful} if the
 algorithm, when ran from $\mu,$ outputs a colouring. We denote the random colouring of vertices of $G$ by $C$. 
 In Sections~\ref{sc:naive_mistake},~\ref{sc:look-ahead_mistake}, we consider the product measure  
 of $(C,\mu)$ and prove that whp $\mu$ is successful for $C$. Let us show that it implies {\it the existence} of a successful $x$ in the deck whp and consequently an efficient derandomisation. Fix $\delta\in(0,1)$ and suppose that $\mathcal{C}_{\delta}$ is the set of all colourings $C'$ of vertices of $G$ such that $\mathbb{P}(\mu\text{ is successful for } C')\leq\delta$. Note that a successful $x$ exists for every $C'\notin\mathcal{C}_{\delta}$. We get that $\mathbb{P}(C\in\mathcal{C}_{\delta})\to 0$. Indeed, if $\mathbb{P}(C\in\mathcal{C}_{\delta})>\varepsilon$ for some $\varepsilon>0$, then
$$
 \mathbb{P}(\mu\text{ is successful for }C)\leq\delta\mathbb{P}(C\in\mathcal{C}_{\delta})+1-\mathbb{P}(C\in\mathcal{C}_{\delta})\leq 1-\varepsilon(1-\delta)
$$
--- a contradiction.  Therefore, indeed whp (over the distribution of the colouring $C$) there exists a successful subgrid in the deck.\\

Let us observe that whp, the algorithm successfully performs the initial step, i.e., there is no other subgrid isomorphic to  $\mu$. Indeed, the probability of the opposite event is at most
$$
 \sum_x\sum_{y\neq x}\mathbb{P}(\mathcal{B}(y)\mid\mu=x)\mathbb{P}(\mu=x)=\sum_{y:\,y\neq x}\mathbb{P}(x\cong y)\leq \frac{n^d}{r^{k^d}}\to 0
$$
as needed.\\


In Sections~\ref{sc:naive_mistake},~\ref{sc:look-ahead_mistake}, we assume that the position of $\mu$ is fixed --- all probabilities are subject to the event $\mu=x_0$, and all the bounds on such probabilities are uniform over the choice of $x_0$.


\subsubsection{The algorithm}
\label{sc:grid_algorithm}

Let us describe the $i$-th step of the algorithm, $i\geq 1$. By that moment, subgrids $\mu=x_0,x_1,\ldots,x_{i-1}$ are already ``placed''. Then $S_{i-1}=x_0\cup x_1\cup \ldots \cup x_{i-1}$ is the restored part of $G.$ The next subgrid $x_{i}$ the algorithm has to find is an {\it extension} of $S_{i-1},$ i.e., $|x_{i}\cap S_{i-1}|\geq k^d - k^{d-1}$. In particular, there exists $j\leq i-1$ and $\ell\in[d]$ such that $x_i=x_j+(-1)^{\delta}\mathbf{e}_{\ell}$ for a certain $\delta\in\{0,1\}$, where $\mathbf{e}_{\ell}=(\underbrace{0,...,0}_{\ell-1},1,\underbrace{0,...,0}_{d-\ell})$ is the {\it $\ell$-th generator} of the lattice $G$. Thus, for given $j(i),\delta(i),\ell(i)$ (described below), the algorithm has to find a subgrid $H$ in the remaining deck such that $S_{i-1}\cap x_i$ is isomorphic to the respective part of $H$, verify the uniqueness of $H$ (in some sense), and then place $H$ by setting $x_i\cong H$. Thus, for every step $i$, it remains to define $j(i),\delta(i),\ell(i)$ as well as the process of verification of the uniqueness of $H$.


We first introduce two types of the $i$-th step. 
The type of the $i$-th step is determined uniquely by $x_0$ and $i$. Note that the algorithm does not know $x_0$ (it only knows its colouring), so it determines the type of the step by $i$ and $S_{i-1}$. Suppose that $j(i),\delta(i),\ell(i)$ are already defined.
\begin{itemize}
    \item \textbf{Naive extension.} If there is only one extension $H$ in the remaining deck, then we set $x_i\cong H$, $S_i=x_0\cup x_1\ldots\cup x_i$, remove $H$ from the deck, and proceed to step $i+1.$ If there are at least two extensions in the deck, the algorithm outputs an `ambiguous' value and terminates. The algorithm is designed in such a way that for a valid input, there is always at least one extension available.
    \begin{figure}[h!]
	\centering
	\includegraphics[scale=0.9]{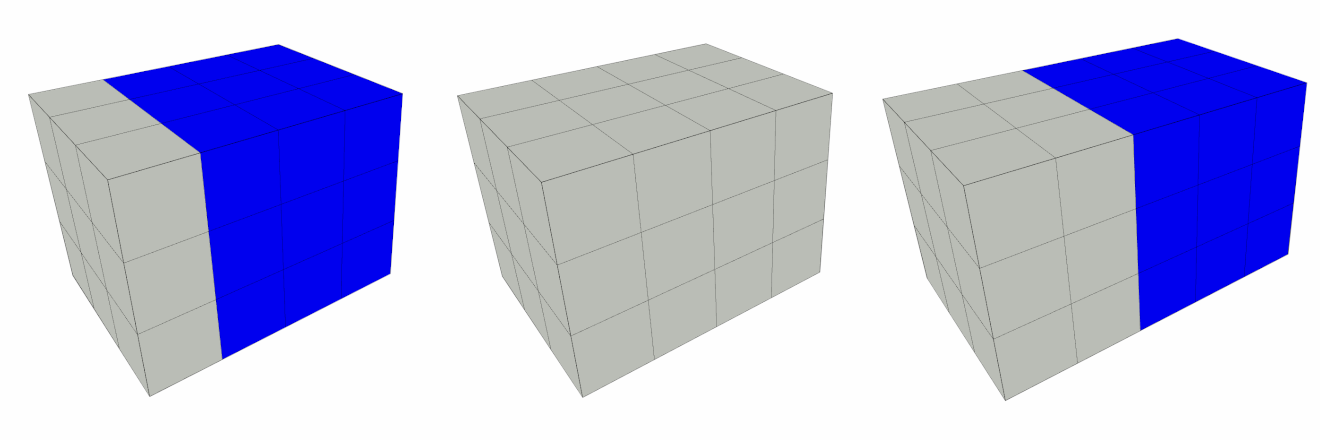}
	\caption{The first two naive extensions at steps 1 and 2; $d=3,$ $k=3,$ the extensions are blue.}
	\label{fig:im1}
\end{figure}
    \item \textbf{Look-ahead extension} requires another generator $\mathbf{e}_{\ell'}$ and another $\delta'\in\{0,1\}$ to be fixed. We will define it later. Let 
    $$
    y_1=x_i,\quad y_2=y_1+(-1)^{\delta'}\mathbf{e}_{\ell'},\quad\ldots,\quad y_k=y_{k-1}+(-1)^{\delta'}\mathbf{e}_{\ell'}.
    $$ 
    The algorithm is designed in a way that, for a valid input, there exist subgrids $H_1,\ldots,H_k$ in the remaining deck such that, for every $j\in[k]$, the intersection of $S_{i-1}$ extended by $H_1,\ldots, H_{j-1}$ with $y_j$ is isomorphic to the respective part of $H_j$. If there are two such tuples $(H_1,\ldots,H_k)$ {\bf with different $H_1$}, then the algorithm outputs `ambiguous'. Otherwise, we set $x_i\cong H_1$, $S_i=x_0\cup x_1\cup\ldots\cup x_i$, remove $H_1$ from the deck, and proceed to step $i+1.$   
\end{itemize}

Let us finally describe the algorithm. As soon as $x_0$ is ``placed'', remove it from the deck and set $\mathcal{D}_1=\mathcal{D}\setminus \{x_0\}.$ Further, let $S_0=x_0,$ $\ell(0)=1$, $\delta(0)=0$.

\begin{enumerate}
    \item[G1]\label{item:step1} The first $2k$ steps are naive. At every step $i=1,\ldots,2k$, set $j(i)=i-1$, $\ell(i)=\ell(i-1)$, $\delta(i)=\delta(i-1)$ and check whether an extension $H$ with these parameters exists in $\mathcal{D}_i$. If it does not (it happens when we reach the border of the grid), then change the direction of exploration: $\delta(i):=1$, $j(i):=0$. Clearly, it is sufficient to do it only once. As soon as the parameters are fixed, apply the naive extension. If after $2k$ steps, the `ambiguous' value is not returned, then eventually a grid $S_{2k}$ of size $3k\times k \times\ldots\times k$ is restored, and the remaining deck $\mathcal{D}_{2k+1}$ has size $(n-k+1)^d-(2k+1).$
    
    \item[G2] Set $i=2k+1$, $j(i)=i-1$, $\ell(i)=2$, $\delta(i)=0$ and check whether a naive extension $H$ exists in $\mathcal{D}_i$. If it does not, then change the direction of exploration: $\delta(i):=1$. Then, apply the look-ahead extension with $\delta'(i)=2-\delta(2k)$, $\ell'(i)=1$ and switch to the step $i:=2k+2$. At this step, we apply a look-ahead extension to the second corner on the ``same side'' of $S_{2k}$. For convenience, let as assume that $x_0$ is the subgrid of $S_{2k}$ furtherst from $x_{2k}$. Set $j(i)=0$, $\ell(i)=2$, $\delta(i)=\delta(i-1)$, $\ell'(i)=1$, $\delta'(i)=\delta(1)$ and apply the look-ahead extension.      
We obtain $S_{2k+2}$. Finally, we apply $k$ look-ahead extensions to colour the vertices between the corners. For every $i=2k+3,\ldots,3k+2$, set $j(i)=i-1$, $\ell(i)=\ell'(i)=1$, $\delta(i)=\delta'(i)=\delta'(i-1)$ and apply the look-ahead extension. Eventually, we get a grid $S_{3k+2}$ of size $3k\times (k + 1)\times k\times\ldots\times k$.

Apply G2 (two look-ahead extensions in corners and $k$ internal look-ahead extensions) $n-k-1$ more times to recover a grid $S_{2k+(k+2)(n-k)}$ of size $3k\times n\times k\times\ldots\times k$.

    \item[G3] Here, the algorithm restores a grid $n\times n\times k\times\ldots\times k$ in a similar way as above: $n-3k$ times apply a pair of look-ahead extensions with $\ell(i)=1$ to two corners, and then fill subgrids between the corners via $n-2k$ internal look-ahead extensions. 
    \item[G4] Apply four look-ahead extensions to four corners of the grid (see Figure~\ref{fig:im2}) and fill the ``frame'' between them in a similar way as in parts G2 and G3 of the algorithm (see Figure~\ref{fig:im3}). Further, in a similar way, fill the remaining $(n-2k)^2$ subgrids of size $1\times1\times 1\times k\times\ldots\times k$ within the frame by performing $(n-2k)^2$ look-ahead extensions (see Figure~\ref{fig:im4}). After applying these sequence of extensions $n-k$ times, we get a restored grid of size $n\times n\times n\times k\times\ldots\times k.$ The further process is similar. As soon as a grid of size $\underbrace{n\times\ldots\times n}_{s}\times \underbrace{k\times\ldots\times k}_{d-s}$ is restored, we apply $2^s$ look-ahead extensions to its corners and, further, restore the inner $n^s-(2k)^s$ subgrids.
\end{enumerate}
\begin{figure}
     \centering
     \begin{subfigure}[b]{0.4\textwidth}
         \centering
        \includegraphics[scale=0.4]
        {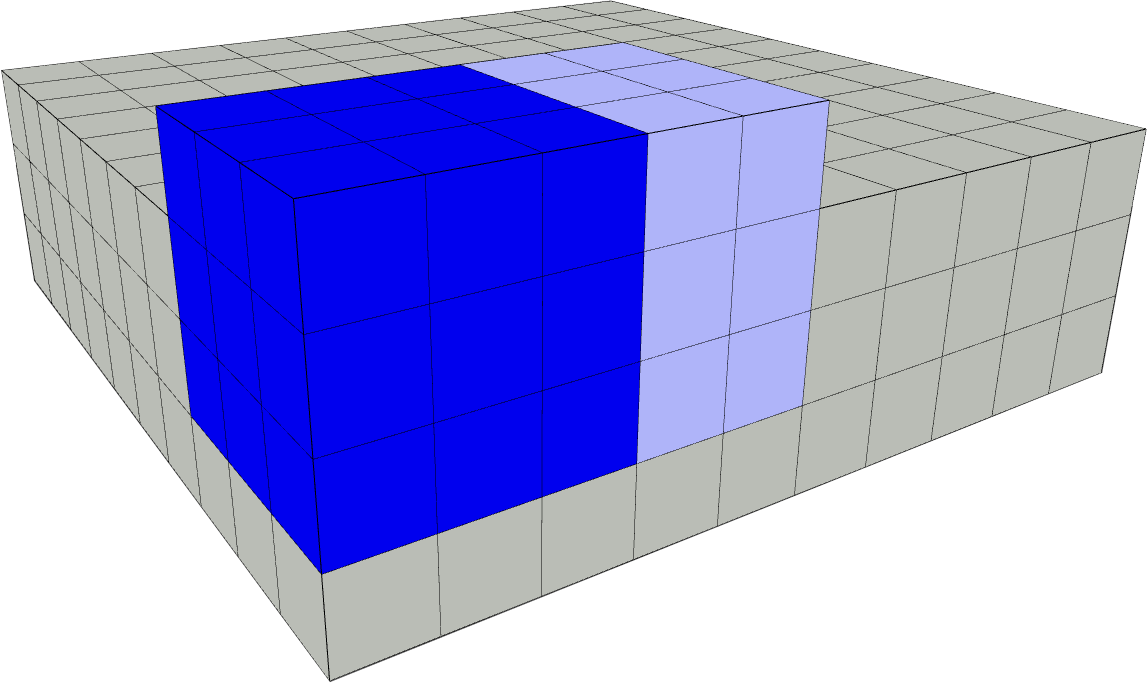}
        \caption{Look-ahead extension in the corner.}
        \label{fig:im2}
     \end{subfigure}
     \hspace{2.4cm}
     \begin{subfigure}[b]{0.3\textwidth}
        \centering
	\includegraphics[scale=0.4]
         {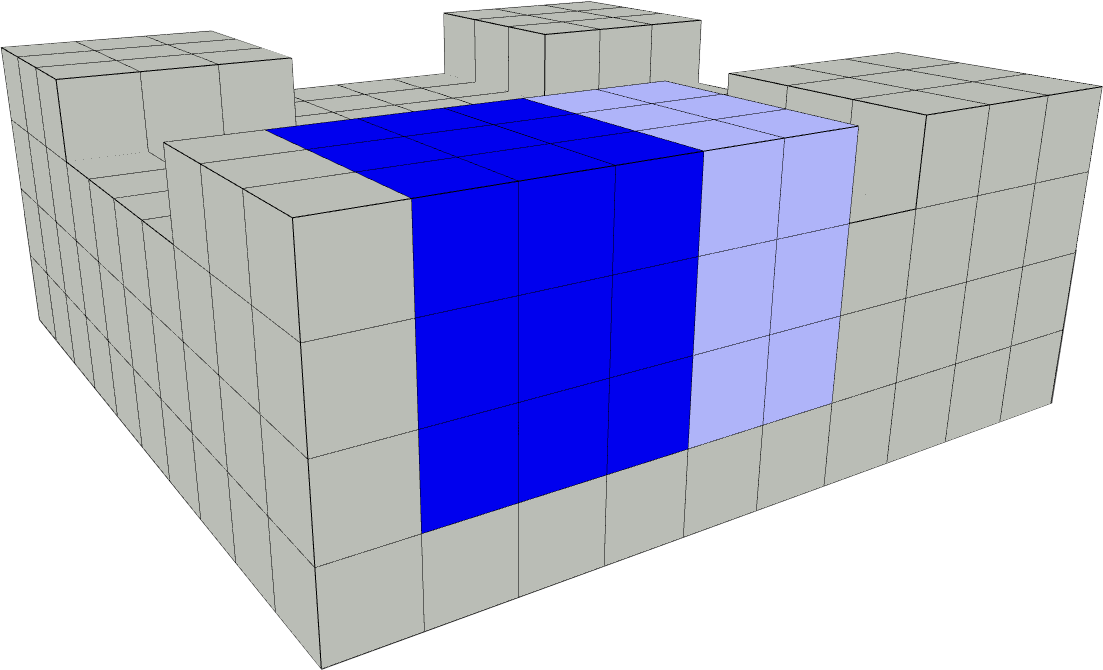}
	\caption{Internal look-ahead extension: filling the frame.}
	\label{fig:im3}
     \end{subfigure}
     
     \begin{subfigure}[b]{0.3\textwidth}
         \centering
        \includegraphics[scale=0.5]{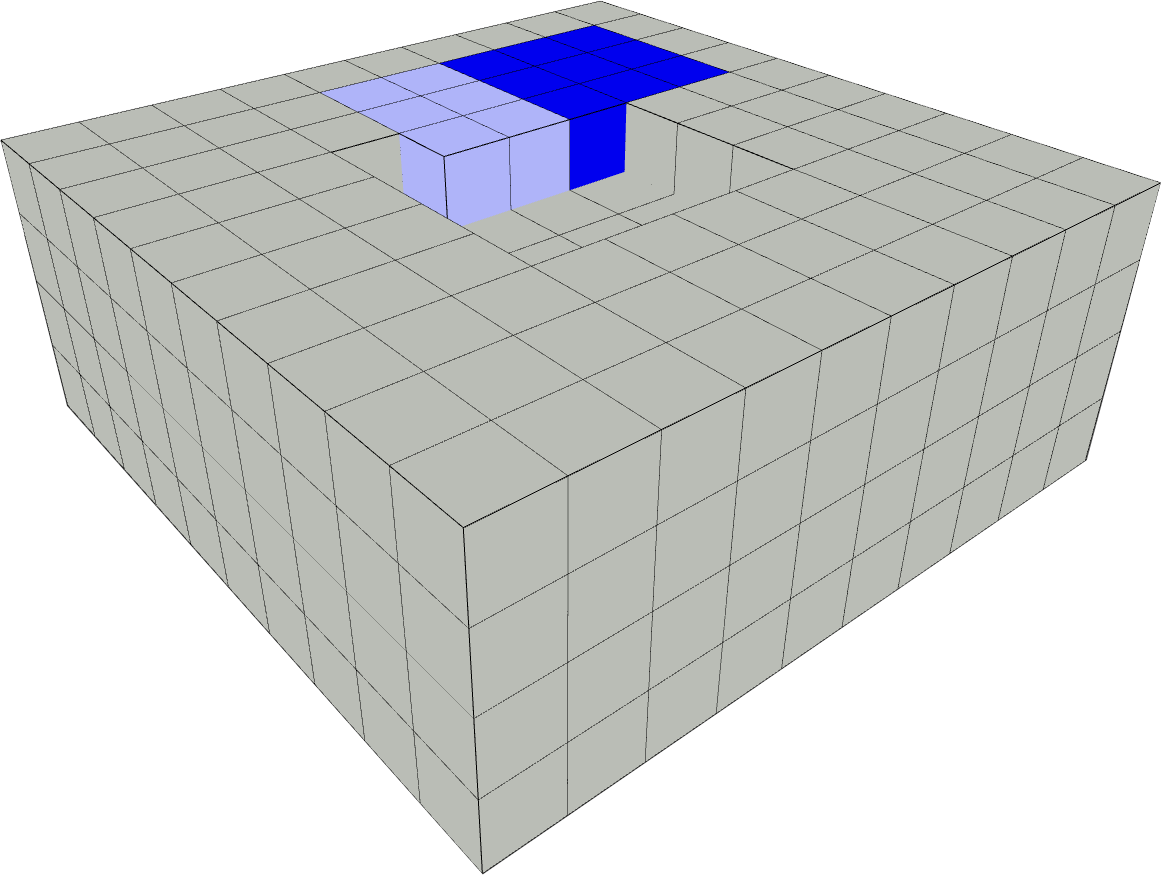}
        \caption{Internal look-ahead extension within the frame.}
        \label{fig:im4}
     \end{subfigure}
        \caption{Look-ahead extensions for $d=3,$ $k=3$; an extension subgrid is blue, subgrids that we place `ahead' are light-blue}
        \label{fig:three graphs}
\end{figure}
Note that, for every $i\leq 2k$, we have $|x_i\cap S_{i-1}|=k^d-k^{d-1}$. Also, for every $i>2k$, the algorithm performs a look-ahead extension so that, for every $j\in\{2,\ldots,k\}$, we have that the intersection of $S_{i-1}$ extended by $H_1,\ldots, H_{j-1}$ with $y_j$ has at least $k^d-k^{d-2}$ elements.
 A successful run of the algorithm performs $2k$ naive extensions and $(n-k+1)^d-2k-1$ look-ahead extensions. In the next two sections, we prove that whp the algorithm outputs the entire colouring.

\subsubsection{Naive extensions are not ambiguous}
\label{sc:naive_mistake}

In this section we prove that whp the algorithm successfully performs all $3k$ naive extensions. Fix $x\in[n-k+1]^d$. Let $\mathrm{NAI}=\mathrm{NAI}(x)$ be the event that there is $y\neq x$ in $[n-k+1]^d$ such that either $x[\mathbf{e}_1]\cong y[\mathbf{e}_1]$ or $x[-\mathbf{e}_1]\cong y[-\mathbf{e}_1]$. Clearly, for every $i\in[3k]$, the initial position $\mu=x_0$ identifies $x$ such that the `ambiguous' value returned at step $i$ implies $\mathrm{NAI}(x)$. Thus, it is sufficient to prove the following. 
\begin{claim}
    $\mathbb{P}(\mathrm{NAI})=\frac{1}{k}\exp\left[-\Omega(\sqrt{\log n})\right]=o\left(\frac{1}{k}\right).$
\label{cl:naive_e}
\end{claim}
\begin{proof}
For a fixed $y\neq x$, we have that $\mathbb{P}(x[\mathbf{e}_1]\cong y[\mathbf{e}_1])=r^{-(k^d-k^{d-1})}$. By the union bound,
$$
 \mathbb{P}(\mathrm{NAI})\leq 2n^d  r^{-(k^d-k^{d-1})}=\frac{1}{k}\exp\left[-\Omega\left(\sqrt{\log n}\right)\right]=o\left(\frac{1}{k}\right)
$$
due to \eqref{eq:helpful}.

\end{proof}

\subsubsection{Look-ahead extensions are not abmiguous}
\label{sc:look-ahead_mistake}

In this section we prove that whp the algorithm successfully performs all $(n-k+1)^d-2k-1$ look-ahead extensions. Fix $x\in[n-k+1]^d$. Let $\mathrm{LC}=\mathrm{LC}(x)$ be the event that 
there are $\ell,\ell'\in[d]$, $\delta,\delta'\in\{0,1\}$ and a path $x_1\ldots x_k$ in $\mathcal{G}$ such that 
\begin{itemize}
\item[C1] $x_1\neq x+(-1)^{\delta'}\mathbf{e}_{\ell'}$ (i.e., the extension is not unique),
\item[C2] for every $i\in[k]$, $\,(x+(i-1)(-1)^{\delta}\mathbf{e}_{\ell})[(-1)^{\delta'}\mathbf{e}_{\ell'}]\cong x_i[-(-1)^{\delta'}\mathbf{e}_{\ell'}]$ (i.e. the wrong extension fits).
\end{itemize}
Clearly, for every $i>3k$ such that a look-ahead extension is applied in a corner, the initial position $\mu=x_0$ identifies $x$ such that the `ambiguous' value returned at step $i$ implies $\mathrm{LC}(x)$. 

In the same way we define the event $\mathrm{LI}=\mathrm{LI}(x)$ describing a failure of an internal look-ahead extension: there are $\ell,\ell'\in[d]$, $\delta,\delta'\in\{0,1\}$ and a path $x_1\ldots x_k$ in $\mathcal{G}$ such that 
\begin{itemize}
\item $x_1\neq x+(-1)^{\delta}\mathbf{e}_{\ell}\,\,$ while $\,\,x[(-1)^{\delta}\mathbf{e}_{\ell}]\cong x_1[-(-1)^{\delta}\mathbf{e}_{\ell}]$,
\item for every $i\in[k]$, $\,y_i[(-1)^{\delta'}\mathbf{e}_{\ell'}]\cong x_i[-(-1)^{\delta'}\mathbf{e}_{\ell'}]$
\end{itemize}
for the blue path $y_1=x-(-1)^{\delta'}\mathbf{e}_{\ell'}+(-1)^{\delta}\mathbf{e}_{\ell}\ldots y_k=x-(-1)^{\delta'}\mathbf{e}_{\ell'}+(k-1)(-1)^{\delta}\mathbf{e}_{\ell}$ in $\mathcal{G}$. 

Finally, let $\mathcal{R}$ be the event that there are no rainbow paths $x_1\ldots x_{\ell}$ in $\mathcal{G}$ such that $\ell\leq 3k$ and $x_{\ell}\in\mathcal{N}(x_1)$. Let us recall that in the definition of $\mathcal{G}$ we fix the ``direction'' of edges (i.e., the vector $\mathbf{e}$ is defined as $\mathbf{e}=(1,0,\ldots,0)$), though clearly Lemma~\ref{lm:lattices_main} holds true for any choice of the direction, and there are constantly many choices. So we treat $\mathcal{R}$ as a stronger event that claims the absence of rainbow paths for any choice of the direction. Let us recall that a successful run of the algorithm performs $O(n^d)$ look-ahead extensions, $O(n)$ of which are applied in corners. Due to Claim~\ref{lm:lattices_main}, it is sufficient to prove

\begin{claim}
    $\mathbb{P}(\mathrm{LC}\wedge\mathcal{R})=o\left(n^{-d+\varepsilon}\right)=o\left(n^{-1}\right)\,$ and
    $\,\mathbb{P}(\mathrm{LI}\wedge\mathcal{R})=o\left(n^{-d}\right).$
\label{cl:look_e}
\end{claim}
\begin{proof}
    Either of $\mathrm{LC}\wedge\mathcal{R}$, $\mathrm{LI}\wedge\mathcal{R}$ implies the existence of a ``wrong'' path $x_1\ldots x_k$ which is a union of non-overlapping in $G$ rectangular subgrids with biggest sides of lengths $s_1,\ldots,s_{\tau}$. Note that $x_1\ldots x_k$ can be treated as a rainbow path in $\mathcal{G}$ with the above rectangular subgrids represented by $\tau$ blue subpaths (again, we might note that the actual `direction' of these paths may not coincide with $\mathbf{e}$ in the definition of $\mathcal{G}$, though the definition remains the same for any choice of $\mathbf{e}$). We bound from above probabilities of both events $\mathrm{LC}\wedge\mathcal{R}$, $\mathrm{LI}\wedge\mathcal{R}$ using the union bound over the choice of the colouring of a rainbow path. Let us fix such a colouring (in other words, the numbers $s_1,\ldots,s_{\tau}$ are fixed). Clearly, $s_1+\ldots+s_{\tau}-(k-1)\tau=k$. Fix any vertices $x_1,\ldots,x_k$ in $\mathcal{G}$ so that the sequential non-overlapping subpaths of $x_1\ldots x_k$ of lengths $s_1-k,\ldots,s_{\tau}-k$ are blue paths in $\mathcal{G}$. Colour edges of $x_1\ldots x_k$ that belong to these subpaths in blue, and all the other edges --- in red. In the case of an extension in a corner, the probability that $x_1$ fits is exactly $r^{-(k^d-k^{d-1})},$ while, for an internal extension, it equals $r^{-(k^d-k^{d-2})}$. For all the other $i\leq k$, the probability that $x_i$ fits depends on the colour of the edge $\{x_{i-1},x_i\}$ in the path: if the edge is red, then this probability is  $r^{-(k^d-k^{d-2})}$; if $\{x_{i-1},x_i\}$ is blue, then the probability equals $r^{-(k^{d-1}-k^{d-2})}$. We then get

    \begin{align*}
    \mathbb{P}(\mathrm{LI}\wedge\mathcal{R}) &\leq
        \sum_{\tau=1}^k 4d^2 k^{\tau} n^{\tau d} r^{-\tau(k^d-k^{d-2})}r^{-(s_1+\ldots+s_{\tau}-k\tau)(k^{d-1}-k^{d-2})} \\
        &=4d^2r^{-k(k^{d-1}-k^{d-2})} \sum_{\tau=1}^k \left(kn^{ d}\right)^{\tau} r^{-\tau(k^d-k^{d-1})}\\
        &= (4d^2+o(1))k n^{d}r^{-2(k^d-k^{d-1})}=o(n^{-d})
    \end{align*}
    due to \eqref{eq:helpful}. Finally,
        \begin{align*}
        \mathbb{P}(\mathrm{LC}\wedge\mathcal{R}) &\leq
        \sum_{\tau=1}^k 4d^2 k^{\tau} n^{\tau d} r^{-\tau(k^d-k^{d-2})}r^{-(s_1+\ldots+s_{\tau}-k\tau)(k^{d-1}-k^{d-2})} r^{k^{d-1}-k^{d-2}} \\
        &= (4d^2+o(1))k n^{d}r^{-2k^d+3k^{d-1}-k^{d-2}}\\
        &=o\left(r^{-k^d+2k^{d-1}-k^{d-2}}\right)=o\left(r^{-k_{\mathrm{th}}^d+2k_{\mathrm{th}}^{d-1}}\right)=o\left(n^{-d+\varepsilon}\right).
    \end{align*}
\end{proof}

\subsubsection{Verification}
\label{sc:verification}

Let us first explain the reason, why we need the verification runs of the algorithm. Let us assume that, after the first run, the algorithm outputs a colouring $C'$ of $G$, and that there is another colouring $C''$ with the same deck. Since the algorithm is designed in such a way that there is no step when at least two subgrids (with their neighbourhoods) fit the current position, it means that there is a step $i$ such that $S_{i-1}$ is a subgraph of both $C'$ and $C''$, while the border is reached in exactly one of the two colourings. In other words, in order to place $x_i$, we have to change the direction $\delta(i)$ for exactly one of the two colourings. Clearly, the algorithm is designed in such a way that, at this step, the direction remains the same, and so the colouring where the border is reached earlier is $C''$. 

We then run the same algorithm with different tuple of directions initiated in G1, G2, G3, and each of the $d-2$ runs of G4, i.e. we change the value of vector $\overrightarrow{\delta}=(\delta(1),\delta(2k+1),\delta(2k+(k+2)(n-k)+1),\ldots)$ and run the algorithm for all possible values of this vector. Due to Claim~\ref{cl:naive_e} and Claim~\ref{cl:look_e}, whp (in product measure) all the verification runs are successful as well. Note that the position of $x_0$ is different in $C'$ and $C''$. Thus, there is a border that closer to $x_0$ in $C'$, then in $C''$. It immediately implies that for any colouring $\tilde C$ other than the inital random colouring $C$ with the same deck, there exists a choice of the value of $\overrightarrow{\delta}$ such that the algorithm reaches the border of the initial random colouring $C$ faster than the border of $\tilde C$. But the event that this run of the algorithm is successfull contradicts $\mathrm{LC}\wedge\mathcal{R}$. Thus, due to Claim~\ref{cl:look_e}, whp all the verification runs output $C$ and there is no other colouring of the grid with the same deck.

\subsubsection{Running time and efficient derandomisation}
\label{sc:grids_time}

Let us first prove that the randomised algorithm can be run in time $O((k(n-k+1))^d\log n)$  for asymptotically almost all colourings, which is linearithmic in the input size. After that we will show that its derandomisation does not influence significantly the running time.\\

We shall observe that, if $k\geq (3d\log_r n)^{1/d}$, then whp there are no red edges in $\mathcal{G}$ (for any choice of $\mathbf{e}$). Indeed, the opposite event has probability at most
$$
n^{2d} r^{-(k^d-k^{d-1})}  \leq
 \exp[2d\ln n-(1-o(1))k^d\ln r]
 \leq
 \exp[-(1-o(1))d\ln n]=o(1).
$$
So, whp no look-ahead extensions are actually needed in this case.\\


Let us {\it specify} certain subgraphs of a $k$-grid $H_k^d$. The set of specified subgraphs is closed under union, and the minimal specified subgraphs are all $2d$ parallepipeds with one side of size $k-1$ and the others of size $k$. A {\it key} is a coloured specific subgraph. Let $\mathcal{K}$ be the set of all keys. Clearly, $|\mathcal{K}|=O(r^{k^d})$.

Let us create a map $\varphi$ that assigns to every $K\in\mathcal{K}$ the multiset of grids in the deck containing $K$ (an orientation might be preserved). It is easy to see that $\varphi$ can be computed in linear time $O((k(n-k+1))^d)$ since it is sufficient to run once through the deck.

Now, we run the algorithm described in Section~\ref{sc:grid_algorithm}. It has $(n-k+1)^d-1$ steps. Let us estimate the cost of a single step $i\geq 1$. At a naive step we find the respective key $K_i$, which is a coloured minimal specified subgraph isomorphic to $S_{i-1}\cap x_i$. Observe that whp $|\varphi(K_i)|=1$ implying $O(k^d)$ time bound. Assume now that the step $i$ is look-ahead. Without loss of generality, we may assume that $k<(3d\log_r n)^{1/d}$.
At a look-ahead step we perform $k$ substeps, each time extending suitable $\mathcal{G}$-paths. At every such substep we go over these paths, find for each one the respecite key $K_{i,j}$, $j\in[k]$, and then replace each path $P$ by a multiset of all paths $Px$, $x\in\varphi(K_{i,j})$. In order to prove that this look-ahead step can be computed in $O(k^{d}\log n)$ time, it is sufficient to show that, for every $j$, the number of suitable $\mathcal{G}$-paths is $O(\log n/k)$. It follows immediately from the following claim, completing the proof of the linearithmic time bound for the randomised algorithm.

Fix $x\in[n-k+1]^d$, $j\in[k]$, $\ell,\ell'\in[d]$,  and $\delta,\delta'\in\{0,1\}$. Let $\Psi_j(x)$ be the number of paths $x_1\ldots x_j$ in $\mathcal{G}$ such that properties C1, C2 from the definition of the event $\mathrm{LC}(x)$ from Section~\ref{sc:look-ahead_mistake} hold.

\begin{claim}
Let $c>0$ be large enough. Whp for all $x,j$, $|\Psi_j(x)|\leq c\sqrt{\log n}$.
\label{cl:few_wrong_paths}
\end{claim}

Note that a similar claim holds true for internal extensions as well (though the probability bounds for internal extensions are even better), thus we omit its proof for the sake  of brevity and avoiding repetitions.\\



{\it Proof of Claim~\ref{cl:few_wrong_paths}}. Due to Lemma~\ref{lm:lattices_main}, it is sufficient to prove that the maximum number of self-non-overlapping and disjoint (in $G$) paths $x_1\ldots x_j$ in $\mathcal{G}$ satisfying C1 and C2 is at most $c\sqrt{\log n}$. Let $\psi=\lceil c\sqrt{\log n}\rceil$. Using the bound on $\mathbb{P}(\mathrm{LC}\wedge\mathcal{R})$ from the proof of Claim~\ref{cl:look_e}, we get that the probability that there are $\psi$ such paths $x_1\ldots x_j$ is at most
\begin{align*}
 \biggl(\sum_{\tau=1}^j k^{\tau} n^{\tau d} & r^{-\tau(k^d-k^{d-2})}
 r^{-(j-\tau)(k^{d-1}-k^{d-2})}r^{k^{d-1}-k^{d-2}}\biggr)^{\psi}=\\
 &=\left(r^{-j(k^{d-1}-k^{d-2})}r^{k^{d-1}-k^{d-2}}\sum_{\tau=1}^j \left(k n^{d} r^{-(k^d-k^{d-1})}\right)^{\tau}\right)^{\psi}\\
 &\leq
 \left((1+o(1))k n^{d} r^{-(k^d-k^{d-1})}\right)^{\psi}\leq r^{-(\varepsilon d^{2-1/d}-o(1))\psi\sqrt{\log n}}=o((kn^d)^{-1})
\end{align*}
due to~\eqref{eq:helpful}. The union bound over $x,j$ finishes the proof.  $\quad\quad\quad\quad\quad\quad\quad\quad\quad\quad\quad\quad\quad\Box$\\

Though a direct derandomisation increases the running time by a $(n-k+1)^d$-factor, it can be reduced to a factor of 2. We may clearly assume that $k<(3d\log_r n)^{1/d}$ (otherwise we can just start from an arbitrary $x_0$). Let us fix an integer $\chi$ such that $r^{\chi}\in[(n/k)^d/k,r(n/k)^d/k]$. Obviously, $\chi\in[k^d]$. 
Instead of choosing a uniformly random $\mu$, we start from an arbitrary $\mu$ such that its first $\chi$ vertices are coloured in the first colour. 
Note that, if such a card exists in the deck (we call it {\it almost white}), it can be found in linear time. Also, for a fixed $x_0\in V(\mathcal{G})$, 
\begin{equation}
\mathbb{P}(x_0\text{ is almost white})=r^{-\chi}\in\left[\frac{k^{d+1}}{rn^d},\frac{k^{d+1}}{n^d}\right].
\label{eq:prob_white}
\end{equation}

Let us prove that whp (in the measure induced by $C$) an almost white card exists. Note that it will immediately imply the desired assertion due to~\eqref{eq:prob_white}, the union bound over all cards, Claims~\ref{cl:naive_e},~\ref{cl:look_e}, and the fact that whp there are no overlapping red-adjacent cards (the latter follows from Lemma~\ref{lm:lattices_main}). In particular, whp for {\it every} (not necessarily almost white) first subgrid $\mu=x_0$ the algorithm can not fail when performing an internal extension step due to the $o(n^{-d})$ bound in Claim~\ref{cl:look_e} (the total number of internal extensions over all $x_0$ is $O(n^d)$).

Let $\Phi$ be the number of almost white cards in the deck. Due to~\eqref{eq:prob_white}, $\mathbb{E}\Phi=\Theta(k^{d+1})$. It remains to prove that $\mathrm{Var}\Phi=o((\mathbb{E}\Phi)^2)$ and apply Chebyshev's inequality. For non-overapping cards, the events that cards are almost white are independent. The number of ovelapping cards is $O(k^d n^d)$. Thus,
$$
 \mathrm{Var}\Phi\leq\mathbb{P}(x_0\text{ is almost white})\times O(k^d n^d)=O(k^{2d+1})=o\left((\mathbb{E}\Phi)^2\right)
$$
completing the proof.\\




\begin{remark}
Note that the algorithm presented in Sections~\ref{sc:step-0},~\ref{sc:grid_algorithm} can be run on the deck of subgrids of a randomly coloured torus $\mathbb{Z}_n^d$, and it successfully outputs the colouring whp as well due to arguments similar to Lemma~\ref{lm:lattices_main} and those from Sections~\ref{sc:naive_mistake},~\ref{sc:look-ahead_mistake}. However, in this case we even do not need verification runs to prove uniqueness since the algorithm can not reach the border, and so, if a colouring with the given deck is not unique, the algorithm initiated at any subgrid would give up, since at some step it should find at least two subgrids (with their neighbourhoods) that fit the current position. We shall note that the algorithm of Ding and Liu~\cite{DL} does not work for tori since they heavily use the presence of corners --- they actually start the exploration from a corner while our algorithm first places a random (or minimum) card from the deck.\\
\label{rmk:tori}
\end{remark}

\begin{remark}
In the proof of Theorem~\ref{th:grids} we assumed that orientations of $k$-subgrids are observed. Nevertheless, the `unoriented' case can be treated similarly. Indeed, a $d$-dimensional lattice has $2^d d!=O(1)$ automorphisms, thus, on every step of the algorithm, there at most $2^d d!$ ways to place an extension from the deck. Since constant factors do not affect significantly probability bounds in Sections~\ref{sc:step-0},~\ref{sc:naive_mistake},~\ref{sc:look-ahead_mistake}, we get that the colouring is reconstructible whp even when orientations are not observed.

\label{rmk:unoriented}
\end{remark}

\section{From colour reconstruction to graph reconstruction: proof of Lemma~\ref{lm:reconstruction_bridge}}
\label{sc:bridge_lemma_proof}

Set $k=k(2n)$. Recall that $G_n\sim G(n,1/2)$ is asymmetric whp. Fix a constant $\delta>0$ and assume that $G_n$ is asymmetric and reconstructible from its full $k$-deck with probability at least $\delta$. Let us colour the set of vertices $[2n]$ of the random graph $G_{2n}$ randomly and independently of the choice of edges of $G_{2n}$ in two colours. By the de Moivre--Laplace limit theorem, with probability $\Theta(n^{-1/2})$ (over the uniform probability measure of the colouring), the colour classes have size exactly $n$.\\

We claim that, given a balanced deterministic bipartition $[2n]=V_1\sqcup V_2$, with probability at least $\delta^2-o(1)$
\begin{itemize}
\item[(i)] each $G_{2n}[V_i]$ is reconstructible from its full $k$-deck, \item[(ii)] each $G_{2n}[V_i]$ is non-isomorphic to any other $n$-vertex induced subgraph of $G_{2n}$.
\end{itemize}
We already know that both $G^i:=G_{2n}[V_i]$, $i=\{1,2\}$, are reconstructible with probability at least $\delta^2$. Thus, it is sufficient to show that $G_{2n}$ satisfies (ii) whp. 

Fix $\varepsilon>0$ small enough. Note that whp, by the Chernoff bound, for any two vertices $u,v$ in $G^1$, the number of neighbours of $v$ that are not adjacent to $u$ is $n/4(1+o(1))$. Let us consider the following event $\mathcal{B}$: 
\begin{center}
there exist $U_1\subset V_1$ of size at most $\varepsilon n$ and $U_2\subset V_2$ of the same size such that an isomorphism $\varphi:G^1\to G_{2n}\left[\left(V_1\cup U_2\right)\setminus U_1\right]$ does not preserve vertices of $V_1\setminus U_1$.
\end{center}
Assuming that $\mathcal{B}$ holds, we immediately get using a standard argument that whp this isomorphism moves $\Theta(n^2)$ edges of $V_1\setminus U_1$ (if $v$ moves to $\varphi(v)=u$, then the set $N$ of neighbours of $v$ in $V_1\setminus U_1$ that are not adjacent to $u$ should be not fixed as well: $\varphi(N)\cap N=\varnothing$, and the number of edges induced by $N$ is $\Theta(n^2)$ with probability $1-e^{-\Theta(n^2)}$). The existence of two subgraphs of size $n$ such that an isomorphism between them moves some set that induces  $\Theta(n^2)$ edges to another disjoint set has probability at most $2^{4n}n!\exp(-\Theta(n^2))$. Thus,
$$
 \mathbb{P}(\mathcal{B})\leq
\mathbb{P}\biggl(\exists v,u\in V_1:\,\,
 |N_{G^1}(v)\setminus N_{G^1}(u)|<(1/4-\varepsilon)n\biggr)+2^{4n}n!\exp(-\Theta(n^2))=o(1).
$$
Finally, by  the union bound over all $k<n-1$ and all $n$-sets $U$ having exactly $k$ vertices in common with $V_1$, we get:
\begin{align*}
\mathbb{P}(\exists U\neq V_1\,\,G_{2n}[U]\cong G_{2n}[V_1])&\leq \mathbb{P}(\mathcal{B})+
\sum_{k=n-\varepsilon n}^{n-1} {n\choose k}\frac{n!}{k!} 2^{-k(n-k)-{n-k\choose 2}}\\
&\quad\quad\quad\,\,\,+
 \sum_{k=0}^{n-\varepsilon n} {n\choose k}{n\choose n-k} n! 2^{-{n-k\choose 2}}\\
 &=o(1)+2^{-n(1-o(1))}+2^{-\Theta(n^2)}=o(1).
\end{align*} 

We then make the same conclusion for the uniformly random {\it balanced} 2-colouring $V_1\sqcup V_2$: with probability at least $\delta^2-o(1)$ (over the product measure) it has properties (i) and (ii). Let $\mathcal{G}$ be the set of all asymmetric graphs $G$ on $[2n]$ such that for a uniformly random balanced bipartition $[2n]=V_1\sqcup V_2$, properties (i) and (ii) hold
with probability at least $\frac{1}{2}\delta^2$. We get that $\mathcal{G}$ is non-empty. Take $G\in\mathcal{G}$. Note that a uniformly random 2-colouring conditioned on the event that it is balanced is a uniformly random balanced colouring. 

Consider a uniformly random colouring $V_1\sqcup V_2$ of the vertices of $G$. Assume that the colouring is balanced and that properties (i) and (ii) hold. Since the $k$-deck of $G[V_i]$ is monochromatic, we may reconstruct the entire colouring using the following simple algorithm: for every $i\in\{1,2\}$, consider all monochromatic cards in the deck of colour $i$ and reconstruct $G[V_i]$. Since there is only one representative of the isomorphism class of $G[V_i]$ in $G$, we get that with probability $\Omega(n^{-1/2})$ the colouring of $G$ is reconstructible from its full coloured $k$-deck --- contradiction.

\section{Proof of Theorem~\ref{th:random_graphs_colour}}
\label{sc:random_graphs_colour_proof}

The 0-statement is proven in Section~\ref{intro:rg_colour}. Here we prove the 1-statement. For that, in the next section we discuss properties of longest induced paths in the random graph, the main lemma stated in Section~\ref{sc:colour_claims} is proven in Section~\ref{rg:colour_claims_proof}. We shall use these properties in the reconstruction algorithm presented in Section~\ref{sc:colour_algo}. The algorithm finds in the deck cards that contain a specific asymmetric coloured induced path that has a unique representative in the entire random graph such that all the other vertices are divided into equivalence classes of cardinalities at most 2: two vertices are equivalent if they have equal neighbourhoods in the path.

\subsection{Longest induced paths}
\label{sc:colour_claims}

Let $C$ be a uniformly random colouring of the vertices of $G_n\sim G(n,1/2)$ in $r$ colours, and let $G^C_n$ be the coloured version of $G_n$. Let $k=\lfloor 2\log_2 n\rfloor  +8$. 

We denote by $\ell(G)$ the maximum number of vertices (we call it {\it length} for convenience despite of a slight abuse of the denotation) in an induced path in a graph $G$. Let $\ell^*=\ell(G_n)$. It is known that $\ell^*$ is concentrated in two consecutive points~\cite{DS}: whp
$$
 \ell^*\in\{\lfloor 2\log_2 n+0.9\rfloor,\lfloor 2\log_2 n+1.9\rfloor\}.
$$

We shall use the following assertions.  

\begin{lemma}
Whp $G_n^C$ contains an induced $\ell^*$-path $P$ such that
\begin{itemize}
\item $P$ is not isomorphic (as a coloured graph) to any other induced $\ell^*$-path in $G_n^C$, and does not have non-trivial automorphisms;
\item the number of vertices $u\notin P$ in $G_n$ such that $G_n[V(P)\cup\{u\}]$ is not isomorphic to $G_n[V(P)\cup\{u'\}]$ for any other $u'\notin P$ equals $n(1-o(1))$.
\item there do not exist three different vertices $u,u',u''\notin P$ such that graphs $G_n[V(P)\cup\{u\}]$, $G_n[V(P)\cup\{u'\}]$, $G_n[V(P)\cup\{u''\}]$ are all isomorphic.
\end{itemize}
\label{cl:uniqlue_coloured_path}
\end{lemma}

The proof of Lemma~\ref{cl:uniqlue_coloured_path} is given in Section~\ref{rg:colour_claims_proof}.



\begin{claim}
Whp, for every $u\neq v$, $G_n$ has $(1/4+o(1))n$ neighbours of $u$ that are not adjacent to $v$.
\label{cl:neighborhood_difference}
\end{claim}

This claim is a standard fact about random graphs that follows immediately from the Chernoff bound and the union bound (see, e.g.,~\cite[Lemma 2.2]{Janson}).\\

Let $G^C$ be a (deterministic) $r$-coloured version of an asymmetric graph $G$ on $[n]$ such that $\ell=\ell(G)\leq k-6$, and let $\mathcal{D}$ be the full $k$-deck of $G^C$. Assume that $G^C$ satisfies the assertions of Lemma~\ref{cl:uniqlue_coloured_path} and Claim~\ref{cl:neighborhood_difference}, namely
\begin{itemize}
\item $G^C$ contains an induced $\ell$-path $P$ such that
\begin{itemize}
\item $P$ is not isomorphic (as a coloured graph) to any other induced $\ell$-path in $G^C$ and does not have non-trivial automorphisms;
\item the number of vertices $u\notin P$ in $G$ such that $G[V(P)\cup\{u\}]$ is not isomorphic to $G[V(P)\cup\{u'\}]$ for any other $u'\notin P$ is at least $0.9n$;
\item there do not exist three different vertices $u,u',u''\notin P$ such that $G[V(P)\cup\{u\}]$, $G[V(P)\cup\{u'\}]$, $G[V(P)\cup\{u''\}]$ are all isomorphic;
\end{itemize}
\item for every $u\neq v$, $G$ has at least $0.2 n$ neighbours of $u$ that are not adjacent to $v$.
\end{itemize}

\subsection{The colour reconstruction algorithm}
\label{sc:colour_algo}

Let us show that $C$ is reconstructible from $\mathcal{D}$. Due to Lemma~\ref{cl:uniqlue_coloured_path} and Claim~\ref{cl:neighborhood_difference}, it would immediately imply the 1-statement in Theorem~\ref{th:random_graphs_colour}. We introduce a colour reconstruction algorithm that outputs a graph $\tilde G^{\tilde C}$ isomorphic to $G^C$.\\

\noindent\textsc{Algorithm A}

\smallskip

\noindent\textsc{Input:} a deck $\mathcal{D}$ consisting of ${n\choose k}$ $r$-coloured graphs of size $k$ and an uncoloured graph $G$ on $[n]$ (though the algorithm does not use $G$).
\begin{enumerate}
\item[A1] Find the maximum length $\ell$ of an induced path in $D$ over all $D\in\mathcal{D}$. If $\ell>k-6$, then reject.

\item[A2] {\it Reconstruction of the colouring of a unique specific maximum path $P$.} 
  Find a subdeck $\mathcal{D}_0\subset\mathcal{D}$ of size ${n-\ell\choose k-\ell}$ such that, for some asymmetric coloured $\ell$-path $P$, 
  
  --- each $D\in\mathcal{D}_0$ contains a unique induced isomorphic copy of $P$ as a coloured graph, 
  
  --- any $D\in\mathcal{D}\setminus\mathcal{D}_0$ does not contain an induced isomorphic copy of $P$,

  --- there is no $D\in\mathcal{D}_0$ that contains vertices $u_1^*,u_2^*,u_3^*$ that have equal neighbourhoods in $P$. 
  
  If $\mathcal{D}_0$ does not exist, then reject. Set $\tilde G^{\tilde C}:=P$.
  



\item[A3] {\it Reconstruction of vertices that are distinguishable by $P$.}
  For every subset $U\subset V(P)$, add a vertex $u=u(U)$ to $\tilde G^{\tilde C}$ with edges from $u$ to every vertex from $U$ if and only if in $\mathcal{D}_0$ there are exactly ${n-\ell-1\choose k-\ell-1}$ graphs $D$ satisfying the following property: $D$ contains a vertex $u^*$ and an induced $\ell$-path $P^*$ such that the isomorphism (of coloured graphs) from $P$ to $P^*$ sends $U$ to the neighbourhood of $u^*$ in $P^*$. 
  Colour $u$ in the colour that $u^*$ has (obviously, all such $u^*$ would correspond to the same vertex in $G$, and thus would have the same colour). If the number of 
   vertices added after this step is less than $0.9n$, then go to step A2 and try another $\mathcal{D}_0$. 
  
  We denote the set of these new (distinguishable by $P$) vertices by $\mathcal{U}$. 
  
\item[A4] {\it Reconstruction of edges between vertices distinguishable by $P$.} For every two distinct vertices $u_1,u_2\in\mathcal{U}$, find a graph $D\in\mathcal{D}_0$ that contains copies $u_1^*,u_2^*$ of $u_1,u_2$ (with respect to $P^*$), and draw an edge between $u_1$ and $u_2$ in $\tilde G^{\tilde C}$ if and only if $u_1^*,u_2^*$ are adjacent in $D$.

\item[A5] {\it Reconstruction of vertices that are indistinguishable by $P$.}
  Initially set $\mathcal{U}'=\varnothing$. For every $D\in\mathcal{D}_0$ containing an induced $\ell$-path $P^*$ and a pair of vertices $u_1^*,u_2^*$ such that 

    
    --- a pair $(u_1,u_2)$ of copies  of $u_1^*,u_2^*$ (with respect to $P$) is not yet in $\mathcal{U}'$, 
  
    --- $u_1^*,u_2^*$ have {\it equal} neighbourhoods in $P^*$,
  
  add $u_1,u_2$ to $\tilde G^{\tilde C}$ (preserving the colours, the adjacency between them, and adjacencies between them and $P$). Add $(u_1,u_2)$ to $\mathcal{U}'$.
  


  \item[A6] {\it Identification of witnesses (distinguishing between $P$-indistinguishable vertices).} For every $(u_1,u_2)\in\mathcal{U}'$, find the first graph $D\in\mathcal{D}_0$ such that it contains copies $u_1^*,u_2^*$ of $u_1,u_2$ and a vertex $w^*$ that has a copy $w\in\mathcal{U}$ (with respect to $P$) such that $w^*$ is adjacent to $u_1^*$ but not to $u_2^*$. If there is no such $D$, reject. We call the vertex $w$ a {\it witness} for $(u_1,u_2)$. Draw an edge between $u_1$ and $w$ (take care of the colours of $u_1,u_2$ if they are different).

   \begin{figure}[h!]
	\centering
	\includegraphics[scale=0.45]{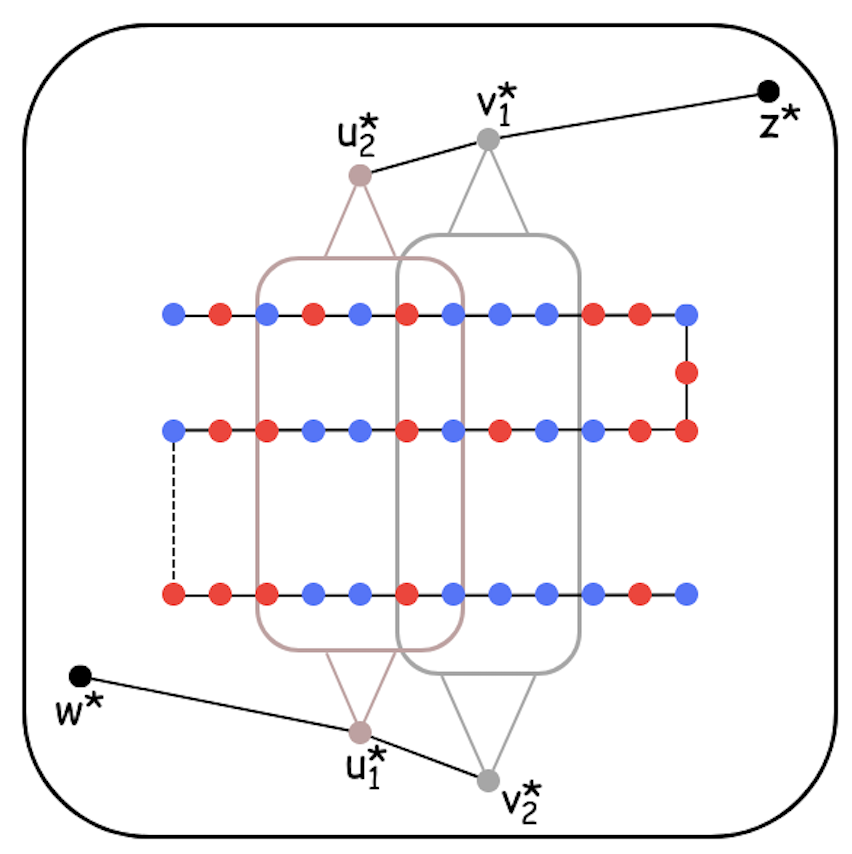}
	\caption{Last step of reconstruction, a single card from the deck $\mathcal{D}_0$ with two pairs of indistinguishable vertices.}
	\label{fig:random_last}
\end{figure}

  \item[A7] {\it Reconstruction of edges between $P$-distinguishable and $P$-indistin\-gui\-sha\-ble vertices.} For every $(u_1,u_2)\in\mathcal{U}'$, its witness $w\in U$, and every $u\in\mathcal{U}$ such that $u\neq w$, find a graph $D\in\mathcal{D}_0$ such that it contains an induced $\ell$-path $P^*$ and copies $w^*,u^*,u_1^*,u_2^*$ of $w,u,u_1,u_2$ (with respect to $P^*$) respectively (in particular, $w^*\sim u_1^*$ and $w^*\nsim u_2^*$). For $j\in\{1,2\}$, draw an edge between $u_j$ and $u$ in $\tilde G^{\tilde C}$ if and only if there is an edge between $u_j^*$ and $u^*$.

  \item[A8] {\it Reconstruction of edges between pairs of $P$-indistinguishable vertices.} For every two different pairs $(u_1,u_2),(v_1,v_2)\in\mathcal{U}'$ and their witnesses $w=w(u_1,u_2)$, $z=w(v_1,v_2)$, find a graph $D\in\mathcal{D}_0$ such that it contains copies $u_1^*,u_2^*$, $v_1^*,v_2^*$, $w^*,z^*$ of these six vertices, see Figure~\ref{fig:random_last}. For every $i\in\{1,2\}$, $j\in\{1,2\}$, draw an edge between $u_i$ and $v_j$ in $\tilde G^{\tilde C}$ if and only if there is an edge between $u_i^*$ and $v_j^*$.
\end{enumerate}

Obviously, for $G^C$, Algorithm A accepts the respective input and outputs a coloured graph. 
From the description of the algorithm it is clear that, as soon as it outputs a coloured graph, it is isomorphic to $G^C$ and there is no other colouring of $G$ with the same deck. Indeed, assume that there is $C'$ such that $G^{C'}$ and $G^C$ are not isomorphic, while they have the same full $k$-deck. If these two colourings coincide on the specific $P$ distilled by the algorithm applied to $\mathcal{D}(G^C)$ at step A2 (i.e. $G^C[V(P)]=G^{C'}[V(P)]$), then clearly the entire colourings coincide. Finally, if the colourings of $P$ are different, then $G$ has a non-trivial automorphism since steps A3--A8 actually reconstruct $G$ --- a contradiction.



\subsection{Proof of Lemma~\ref{cl:uniqlue_coloured_path}}
\label{rg:colour_claims_proof}

For $\ell\in\mathbb{N}$, let $X_{\ell}$ be the number of induced $\ell$-paths in $G_n$. Let 
\begin{equation}
\ell_0=\lfloor 2\log_2 n+0.9\rfloor.
\label{eq:ell_0_definition}
\end{equation}
Due to~\cite{DS}, the following is true
\begin{itemize}
\item $\mathbb{E}X_{j+1}/\mathbb{E}{X}_{j}=\Theta\left(\frac{1}{n}\right)$ for $j\in\{\ell_0,\ell_0+1\}$;
\item $\mathbb{E}X_{\ell_0+2}<n^{-0.8}$ for all $n$ large enough; 
\item for $j\in\{\ell_0,\ell_0+1\}$, if $\mathbb{E}X_j\to\infty$, then $\mathrm{Var}X_j=o((\mathbb{E}X_j)^2)$;
\item $\mathbb{E}X_{\ell_0}\to\infty$. 
\end{itemize}
Therefore, whp $\ell^*\in\{\ell_0,\ell_0+1\}$. In particular, the following two situations are possible:
\begin{enumerate}
\item $\mathbb{E}X_{\ell_0+1}\to\infty$, $\ell^*=\ell_0+1$ and $\mathbb{E}X_{\ell_0+1}=O(n^{0.2})$;
\item $\mathbb{E}X_{\ell_0}=O(n)$ and $\mathbb{E}X_{\ell_0+1}=O(1)$.
\end{enumerate}
Let us first prove that, in both cases, whp 
\begin{itemize}
\item there are no two vertex-overlapping induced paths of length $\ell_0+1$;
\item for every induced $(\ell_0+1)$-path $P$, the number of vertices $u\notin P$ such that $G_n[V(P)\cup\{u\}]$ is not isomorphic to $G_n[V(P)\cup\{u'\}]$ for any other $u'\notin P$ equals $n(1-o(1))$.
\item for every induced $(\ell_0+1)$-path $P$, there do not exist three different vertices $u,u',u''\notin P$ such that $G_n[V(P)\cup\{u\}]$, $G_n[V(P)\cup\{u'\}]$, $G_n[V(P)\cup\{u''\}]$ are all isomorphic.
\end{itemize}
The expected number of pairs of overlapping induced $(\ell_0+1)$-paths is at most
\begin{multline*}
\mathbb{E}X_{\ell_0+1}\times
 \sum_{j=1}^{\ell_0} {\ell_0+1\choose j}
 {n-\ell_0-1\choose \ell_0+1-j}\frac{(\ell_0+1)!}{2}2^{-j(\ell_0+1-j)-{\ell_0+1-j\choose 2}}=\\
=\left(\mathbb{E}X_{\ell_0+1}\right)^2\times
 \sum_{j=1}^{\ell_0}\frac{ {\ell_0+1\choose j}
 {n-\ell_0-1\choose \ell_0+1-j}}{{n\choose\ell_0+1}}2^{{j\choose 2}}.
\end{multline*}
Letting $F_j=\frac{ {\ell_0+1\choose j}
 {n-\ell_0-1\choose \ell_0+1-j}}{{n\choose\ell_0+1}}2^{{j\choose 2}}$, we get
 $$
  \frac{\partial}{\partial j}\ln\frac{F_{j+1}}{F_j}=\ln 2-\frac{2}{\ell_0+1-j}-\frac{1}{j+1}-\frac{1}{n-2\ell_0+j-1}
$$
implying that $\frac{F_{j+1}}{F_j}$ increases in $j$ on $[1,j^*)$ and decreases on $(j^*,\ell_0-1]$, where $j^*=\ell_0+1-2/\ln 2(1-o(1))$. On the other hand
$$
\frac{F_{\ell_0}}{F_{\ell_0-1}}=2^{\ell_0-1}\frac{4}{\ell_0(n-\ell_0-2)}=n^{1-o(1)},
$$
implying that $F_j$ changes its monotonicity at most ones, and, for large $j$, it increases. Therefore, $F_j\leq\max\{F_1,F_{\ell_0}\}$ for all $j$. Observe that 
$$
F_1\leq \frac{(\ell_0+1)^2}{n-\ell_0}\quad\text{ and }\quad
F_{\ell_0}=\frac{(\ell_0+1)(n-\ell_0-1)2^{-\ell_0}}{\mathbb{E}X_{\ell_0+1}}.
$$
Eventually, we get that the expected number of pairs of overlapping paths is at most
$$
\ell_0(\ell_0+1)\max\left\{(\mathbb{E}X_{\ell_0+1})^2\frac{\ell_0+1}{n-\ell_0},\mathbb{E}X_{\ell_0+1}(n-
\ell_0-1)2^{-\ell_0}\right\}=n^{-0.6+o(1)}
$$
since $\mathbb{E}X_{\ell_0+1}=O(n^{0.2})$
as needed. 

Now, let us fix an ordered tuple of vertices $P=(v_1,\ldots,v_{\ell_0+1})$ and a vertex $u\notin P$. Let $\mathcal{N}$ be the neighbourhood of $u$ in $P$ and $\mathcal{N}'$ be the image of $\mathcal{N}$ under the permutation $v_j\to v_{\ell_0+2-j}$.  Let us say that $u$ is {\it `bad'} with respect to $P$ if 
 there exists a vertex $v\notin V(P)\cup u$ such that its neighbourhood in $P$ is {\it identical} to the neighbourhood $u$, i.e. equals $\mathcal{N}$ or $\mathcal{N}'$. Then  
\begin{equation}
\mathbb{P}(u\text{ is `bad' w.r.t. }P)\leq 1-(1-2^{-\ell_0})^n=O(n^{-1}).
\label{eq:bad_u}
\end{equation}
Thus, the expected number of `bad' $u$  is $O(1)$. By Markov's inequality, the number of such `bad' $u$ is bigger than $\sqrt{n}$ with probability $\pi_0=O(n^{-1/2})$. Thus, the expected number of induced $(\ell_0+1)$-paths having  at least $\sqrt{n}$ `bad' vertices $u$ equals $\pi_0\mathbb{E}X_{\ell_0+1}=O(n^{-0.3})$. By Markov's inequality, whp there are no such paths.

Finally, for an ordered tuple $(v_1,\ldots,v_{\ell_0+1})$ the probability that fixed vertices $u,u',u''$ have identical neighbourhoods in $(v_1,\ldots,v_{\ell_0+1})$ is 
$O(n^{-4})$. Thus, the expected number of tuples $(P,u,u',u'')$, where $P$ is an induced $(\ell_0+1)$-path and vertices $u,u',u''$ have identical neighbourhoods in $P$ equals $O(n^{-0.8})$. By Markov's inequality, whp there are no such tuples.\\

We now consider $\ell_0$-paths. Note that $\ell^*=\ell_0$ with probability bounded away from 0 only when $\mathbb{E}X_{\ell_0}=O(n)$. Thus, we further assume that this is indeed the case. Since $\mathbb{E}X_{\ell_0}\to\infty$ and $\mathrm{Var}X_{\ell_0}=o((\mathbb{E}X_{\ell_0})^2)$, we get that $X_{\ell_0}/\mathbb{E}X_{\ell_0}\stackrel{{\mathbb{P}}}\to 1$. Let
\begin{itemize}
\item $\mu_0$ be the expected number of pairs of induced $\ell_0$-paths having at least 2 common vertices; \item $\mu_1$ be the expected number of induced $\ell_0$-paths $P$ having more than $\sqrt{n}$ `bad' $u$ (w.r.t. $P$); 
\item $\mu_2$ be the expected number of tuples $(P,u,u',u'')$, where $P$ is an induced $\ell_0$-path and vertices $u,u',u''$ have identical neighbourhoods in $P$.
\end{itemize}
Let us show that 
\begin{equation}
\max\{\mu_0,\mu_1,\mu_2\}=o(\mathbb{E}X_{\ell_0})
\label{eq:expectations_small}
\end{equation}
implying 

\begin{claim}
Whp $G_n$ contains an induced $\ell^*$-path $P$ such that
\begin{itemize}
\item any other induced $\ell^*$-path shares with $P$ at most 1 vertex;
\item the number of vertices $u\notin P$ such that $G_n[V(P)\cup\{u\}]$ is not isomorphic to $G_n[V(P)\cup\{u'\}]$ for any other $u'\notin P$ equals $n(1-o(1))$;
\item there do not exist three different vertices $u,u',u''\notin P$ such that graphs $G_n[V(P)\cup\{u\}]$, $G_n[V(P)\cup\{u'\}]$, $G_n[V(P)\cup\{u''\}]$ are all isomorphic.
\end{itemize}
\label{cl:intersections_max_paths}
\end{claim}

\begin{proof}
It remains to show that  $\mathbb{E}X_{\ell_0}=O(n)$ implies \eqref{eq:expectations_small}. In the same way as above, we get
$$
\mu_0\leq(\mathbb{E}X_{\ell_0})^2\sum_{j=2}^{\ell_0-1}F_j'\leq\ell_0(\mathbb{E}X_{\ell_0})^2\max\{F_2',F_{\ell_0-1}'\},
$$
where $F_j'=\frac{ {\ell_0\choose j}
 {n-\ell_0\choose \ell_0-j}}{{n\choose\ell_0}}2^{{j\choose 2}}$. Thus,
$$
\mu_0  \leq \ell_0^2\max\left\{(\mathbb{E}X_{\ell_0})^2\frac{\ell_0^3}{(n-\ell_0+1)^2},\mathbb{E}X_{\ell_0}(n-
\ell_0)2^{-\ell_0+1}\right\}
=O\left(\frac{\log^5 n}{n}\mathbb{E}X_{\ell_0}\right).
$$
Let us fix an ordered tuple of vertices $P=(v_1,\ldots,v_{\ell_0})$ and a vertex $u$. Clearly, the $O(1/n)$-bound in \eqref{eq:bad_u} holds here as well.
Thus, the expected number of `bad' $u$  is $O(1)$. Let $\pi_1=O(n^{-1/2})$ be the probability that the number of `bad' $u$ with respect to $P$ is bigger than $\sqrt{n}$. We get
$$
\mu_1=\pi_1\mathbb{E}X_{\ell_0}=o(\mathbb{E}X_{\ell_0}).
$$
Finally, in the same way as above, the probability that fixed vertices $u,u',u''$ have identical neighbourhoods in $P$ is 
$O(n^{-4})$. Thus, 
$$
\mu_2=O(n^{-1})\mathbb{E}X_{\ell_0},
$$
as needed.
\end{proof}

We now finish the proof of Lemma~\ref{cl:uniqlue_coloured_path}. Let us expose edges of $G_n$ and fix an induced path $P$ of length $\ell^*$ satisfying the properties listed in Claim~\ref{cl:intersections_max_paths}. Let $\mathcal{P}$ be the set of all induced $\ell^*$-paths other than $P$. From the above, we may assume $|\mathcal{P}|<n^{1.1}$ (by Markov's inequality, it holds whp in $G_n$). 

We colour randomly the vertices of $G_n$ in $r$ colours, and denote the random colouring by $C$. Whp the colouring of $P$ is asymmetric, i.e. the coloured version of $P$ does not have non-trivial isomorphisms. Indeed, the probability that it is symmetric is exactly $r^{-\lfloor\ell^*/2\rfloor}=o(1)$. For any $P'\in\mathcal{P}$, we have that 
$$
\mathbb{P}(P\cong_C P')\leq 2r^{-\ell_0+1}=O(n^{-2})
$$
since $P'$ shares at most 1 vertex with $P$ due to the choice of $P$.
By the union bound, every $P'\in\mathcal{P}$ is not isomorphic to $P$ as a coloured graph whp, completing the proof.

\section{Proof of Theorem~\ref{th:random_graphs_rec}}
\label{sc:random_graphs_rec_proof}

Note that Algorithm A does not use the input graph $G$. Moreover, colours of vertices are beneficial only on step A2, and on all the other steps we just copy colours of vertices from the deck, but do not use them for the graph reconstruction (though we use the fact that $P$ is asymmetric as a coloured graph). In this section, we show how to modify the step A2 by colouring vertices of a longest path $P$ in 8 colours artificially using a 3-tuple of vertices outside $P$ that have a unique neighbourhood in $P$. This colouring would make $P$ asymmetric. For that, we need the following claim.

\begin{claim}
Whp, for every induced $\ell^*$-path $P$ in $G_n\sim G(n,1/2)$, there exist three vertices $u_1,u_2,u_3\notin P$ such that
\begin{itemize}
\item for any $u\in\{u_1,u_2,u_3\}$ and  $u'\notin V(P)\cup\{u\}$, the graphs $G_n[V(P)\cup\{u\}]$ and $G_n[V(P)\cup\{u'\}]$ are not isomorphic,
\item for any other induced $\ell^*$-path $P'$ in $G_n$ that shares with $P$ at most 1 vertex there exists $u\in\{u_1,u_2,u_3\}$ such that, for any vertex $u'\notin P'$, graphs $G_n[V(P)\cup\{u\}]$ and $G_n[V(P')\cup\{u'\}]$ are not isomorphic,
\item the graph $G_n[V(P)\cup\{u_1,u_2,u_3\}]$ is asymmetric.
\end{itemize}
\label{cl:colouring_paths}
\end{claim}

\begin{proof}
Let $\ell\in\{\ell_0,\ell_0+1\}$. For any $\ell$-tuple of vertices $P$ in $[n]$, let us choose arbitrarily six distinct vertices $u^P_1,\ldots,u^P_6$ outside $P$. Let us say that a vertex $u\notin P$ is {\it `bad'} with respect to $P$, if one of the following conditions holds:
\begin{itemize}
\item $G_n[V(P)\cup u]$ has a non-trivial automorphism,
\item $u$ has less than 7 neighbours in $P$,
\item there exists a vertex $v\notin P\cup\{u\}$ that is {\it identical} to $u$ with respect to $P$ (i.e. graphs $G_n[P\cup\{u\}]$ and $G_n[P\cup\{v\}]$ are isomorphic).
\end{itemize}
The probability that both $u^P_1$ and $u^P_2$ are `bad' with respect to $P$ equals $O(n^{-2})$. From Section~\ref{rg:colour_claims_proof}, we know that the expected number of induced $\ell$-paths equals $O(n^{1.2})$. Therefore, the expected number of induced $\ell$-paths $P$ such that for some $j\in\{1,2,3\}$ both $u^P_{2j-1},u^P_{2j}$ are `bad' with respect to $P$ equals $O(n^{-0.8})$. By Markov's inequality, whp there are no such induced $\ell$-paths.

For every $\ell$-tuple $P$, let $\mathcal{U}^P$ be the set of all  3-tuples having a single vertex in each of the three pairs $(u^P_1,u^P_2),$ $(u^P_3,u^P_4)$ and $(u^P_5,u^P_6)$. Obviously, $|\mathcal{U}^P|=8$ for every $P$. 
For two fixed $\ell$-tuples $P,P'$ that share at most one vertex, the probability that there exists $(u_1,u_2,u_3)\in\mathcal{U}^{P}$ such that, for every $i\in\{1,2,3\}$, there is a vertex $u'\notin P'$ such that $G_n[P\cup\{u_i\}]$ and $G_n[P'\cup\{u'\}]$ are isomorphic is at most $O(n^{-3})$. Therefore, the expected number of induced $\ell$-paths $P,P'$ satisfying the above property equals $O(n^{-0.6})$. By Markov's inequality, whp there are no such pairs of induced $\ell$-paths.

Thus, whp for every induced $\ell^*$-path $P$ in $G_n$ there is a tuple $(u_1,u_2,u_3)\in\mathcal{U}^P$ of vertices that are not `bad' with respect to $P$ and ``distinguishes'' $P$ from any other induced $\ell^*$-path that shares with $P$ at most 1 vertex. It remains to show that $H:=G_n[V(P)\cup\{u_1,u_2,u_3\}]$ is asymmetric. Let $f$ be an automorphism of $H$. Then $f$ either preserves $P$ or turns it over since every $u_i$ has at least 7 neighbours in $P$. But then, since all $u_i$ are not `bad', they all `see' $P$ differently. Thus, $f$ preserves all $u_i$. But then $f$ also preserves all vertices in $P$ since $G_n[V(P)\cup u]$ are asymmetric.
\end{proof}

Let $k=\lfloor 2\log_2 n\rfloor+11.$ Let $G$ be a (deterministic)  graph on $[n]$ such that the maximum length (number of vertices) of an induced path in this graph is $\ell=\ell(G)\leq k-9$, and let $\mathcal{D}$ be the full $k$-deck of $G$. Assume that $G$ satisfies the assertions of Claims~\ref{cl:neighborhood_difference},~\ref{cl:intersections_max_paths},~\ref{cl:colouring_paths}, namely
\begin{itemize}
\item $G$ contains an induced $\ell$-path $P$ and vertices $u_1,u_2,u_3\notin P$ such that
\begin{itemize}
\item any other induced $\ell$-path shares with $P$ at most 1 vertex,
\item the number of vertices $u\notin P$ such that $G[V(P)\cup\{u\}]$ is not isomorphic to $G[V(P)\cup\{u'\}]$ for any other $u'\notin P$ is at least $0.9n$;
\item there do not exist three different vertices $u,u',u''\notin P$ such that graphs $G[V(P)\cup\{u\}]$, $G[V(P)\cup\{u'\}]$, $G[V(P)\cup\{u''\}]$ are all isomorphic,
\item for every $u\in\{u_1,u_2,u_3\}$ and every vertex $u'\notin V(P)\cup\{u\}$, graphs $G[V(P)\cup\{u\}]$ and $G[V(P)\cup\{u'\}]$ are not isomorphic,
\item for any other induced $\ell$-path $P'$ in $G$, there exists a vertex $u\in\{u_1,u_2,u_3\}$ such that, for any vertex $u'\notin P'$, graphs $G[V(P)\cup\{u\}]$ and $G[V(P')\cup\{u'\}]$ are not isomorphic,
\item the graph $G[V(P)\cup\{u_1,u_2,u_3\}]$ is asymmetric;
\end{itemize}
\item for every $u\neq v$, $G$ has at least $0.2 n$ neighbours of $u$ that are not adjacent to $v$.
\end{itemize}

Let us now introduce an algorithm that successfully reconstructs from $\mathcal{D}$ a graph isomorphic to $G$. This algorithm is designed in a way such that it succeeds only when there is no other $G'\not\cong G$ with $\mathcal{D}(G')=\mathcal{D}$. Due to Claims~\ref{cl:neighborhood_difference},~\ref{cl:intersections_max_paths},~\ref{cl:colouring_paths}, it would immediately imply the 1-statement in Theorem~\ref{th:random_graphs_rec}. This algorithm does essentially the same steps as Algorithm A but the step A2 since we can not directly use colours. Thus, we replace this step with B2 that identifies $P$ via its witnesses $u_1,u_2,u_3$.\\

\noindent\textsc{Algorithm B}

\smallskip

\noindent\textsc{Input:} a deck $\mathcal{D}$ consisting of ${n\choose k}$ graphs of size $k$.
\begin{enumerate}
\item[B1] Find the maximum length $\ell$ of an induced path in $D$ over all $D\in\mathcal{D}$. If $\ell>k-9$, then reject.
\item[B2] {\it Identify a unique specific longest path.} Find a subdeck $\mathcal{D}_0\subset\mathcal{D}$ of size ${n-\ell-3\choose k-\ell-3}$ such that each $D\in\mathcal{D}_0$ contains an induced $\ell$-path $P=P(D)$ and three vertices $u^D_1,u^D_2,u^D_3\notin P$ satisfying the following requirements:

--- there are no induced $\ell$-paths other than $P(D)$ in every $D\in\mathcal{D}_0$,

--- for any two $D,D'\in\mathcal{D}_0$ and for every $i\in\{1,2,3\}$, vertices $u^D_i$ and $u^{D'}_i$ are identical with respect to $P(D)$ and $P(D')$ respectively,

--- for any $D\in\mathcal{D}_0$, every $u\in\{u^D_1,u^D_2,u^D_3\}$, and every $u'\in V(D)\setminus (V(P)\cup u)$, graphs $D[V(P)\cup u]$ and $D[V(P)\cup u']$ are not isomorphic,

--- for any $D\in\mathcal{D}_0$, the graph $D[V(P)\cup\{u^D_1,u^D_2,u^D_3\}]$ is asymmetric,
  
--- there is no $D'\in\mathcal{D}\setminus\mathcal{D}_0$ containing an induced $\ell$-path $P'$ and three vertices $u'_1,u'_2,u'_3\notin P'$ such that, for some $D\in\mathcal{D}_0$ and every $i\in\{1,2,3\}$, $u^D_i$ and $u^{D'}_i$ are identical with respect to $P(D)$ and $P(D')$ respectively,

--- there is no $D\in\mathcal{D}_0$ that contains three vertices with identical neighbourhoods in $P(D)$.

If $\mathcal{D}_0$ does not exist, then reject. Let $\tilde G:=G_0$ be a graph isomorphic to $D[V(P)\cup\{u^D_1,u^D_2,u^D_3\}]$.

\item[B3] Apply steps A3--A8 from Algorithm A with $P$ replaced by $G_0$ (no need to colour vertices).
\end{enumerate}

\section{Discussions}

The tight concentration result in Theorem~\ref{th:grids} could be possibly further improved if we allow other graphs in the deck. Note that the subgrids have $k^d$ vertices. It is easy to see that, for $k$ satisfying the requirement in the 0-statement in Theorem~\ref{th:grids}, it is still whp impossible to reconstruct the colouring from the full $k^d$-deck. In other words, subgraphs other than grids do not help much. On the other hand, the difference between the upper bound (that follows directly from Theorem~\ref{th:grids} since subgrids could be extracted from the full deck) and the lower bound for the reconstruction threshold from the full deck equals $\Theta(k^{d-1})$. It would be interesting to improve the concentration interval or to prove that this is not possible.

Though Theorem~\ref{th:random_graphs_colour} states that, for any asymmetric graph, whp its random colouring is not reconstructible from its full $k$-deck whenever $k\leq\sqrt{2\log_2 n}$, we can not show that this bound can not be improved to any other $o(\log n)$-bound. Also, it is not clear, whether, for any asymmetric $G$ (or at least for $G(n,1/2)$ with high probability), there exists a (sharp) threshold for a reconstruction of a uniformly random colouring of $G$. Same question could be asked for the graph reconstruction of a random graph. We shall note that, for $k\geq 2\log_2n$, whp there are isomorphism classes of graphs of size $k$ that are not presented in the full $k$-deck of $G(n,1/2)$ --- and this is a crucial observation for our argument to work. However, when $k\leq c\log_2 n$ for sufficiently small $c>0$ (we believe that any $c<2$ should be enough --- see, e.g.,~\cite{KSZ}), then any
isomorphism class of graphs of size $k$ has many representatives in the full $k$-deck (proportional to the index of the respective automorphism group in the symmetric group). Thus, our approach can not be used to show that the random graph is reconstructible for such $k$. Though we do not know whether the constant factor in the first order term in our upper bound for the (possible) threshold can be refined, the second order term can be certainly improved, and we did not try to optimise it.

The {\it $\ell$-reconstruction number} $f_G(\ell)$ of a graph $G$ is the minimum size of an $(n-\ell)$-deck $\mathcal{D}$ that is sufficient to reconstruct $G$, i.e. there is no other graph such that its full $(n-\ell)$-deck contains $\mathcal{D}$ as a submultiset. Bollob\'{a}s~\cite{Bol_reconstruction} proved that actually any three graphs from the full $(n-1)$-deck are enough to reconstruct $G_n\sim G(n,1/2)$ whp. The result of Bollob\'{a}s immediately implies that whp $f_{G_n}(1)=3$. 
Our reconstruction algorithm requires the entire deck. It would be interesting to estimate the $\ell$-reconstruction number for $\ell=n-O(\log n)$. 

For a graph property $\mathcal{P}$ (i.e. set of graphs closed under isomorphism), let us say that a graph $G$ having the property $\mathcal{P}$ is {\it weakly $\mathcal{P}$-reconstructible from its (partial) $k$-deck $\mathcal{D}$}~(more details about the weak reconstruction can be found in \cite{Intro_BH}), if there is no graph $G'$ non-isomorphic to $G$ and having the property $\mathcal{P}$ such that the full $k$-deck of $G'$ contains $\mathcal{D}$ as a submultiset. Let $\varepsilon>0$ be a constant, and $k=k(n)\geq(1/2+\varepsilon)n$ be a sequence of integers. Consider the following property $\mathcal{P}$: a graph $G$ on $[n]$ has the property $\mathcal{P}$, if all induced subgraphs of $G$ on $k-1$ vertices are pairwise non-isomorphic and asymmetric. It is known~\cite{Muller} that whp $G_n\sim G(n,1/2)$ has the property $\mathcal{P}$. Spinoza and West~\cite{SW} proved that, for any graph $G$ having the property $\mathcal{P}$, its full $k$-deck contains a submultiset $\mathcal{D}=\mathcal{D}(G)$ consisting of only ${n-k+2\choose 2}$ graphs such that $G$ is weakly $\mathcal{P}$-reconstructible from this $\mathcal{D}$. Actually they introduced an efficient algorithm that outputs a graph isomorphic to $G\in\mathcal{P}$ provided with exactly the partial deck $\mathcal{D}(G)$. 
It can be shown that whp in $G_n$ any two disjoint induced $\ell_0$-paths, where $\ell_0$ is defined by \eqref{eq:ell_0_definition}, have unique adjacencies between them, implying that (we of course omit here some technicalities, but the proof strategy is very similar to the proof of Theorem~\ref{th:random_graphs_rec}) the above weak reconstruction result (but for a different property $\mathcal{P}$ such that $G_n\in\mathcal{P}$ whp) holds true for any $k\geq 4\log_2 n+6$ as well. We do not know whether the same is true for $k=2\log_2 n+O(1)$.



\section*{Acknowledgements}

The authors would like to thank Viktor Zamaraev for helpful remarks and valuable comments on the paper.

\end{document}